\documentclass[11pt,a4paper]{article}
\usepackage[vmargin=2.5cm,hmargin=2.2cm]{geometry} 
\usepackage{mathrsfs,amssymb,amsmath,amsfonts}
\usepackage{amsthm}
\usepackage{enumerate}
\usepackage{color}
\usepackage{wrapfig} 
\usepackage{epstopdf}
\usepackage{mathtools}
\usepackage{amsfonts}
\usepackage{graphicx}
\usepackage{amssymb}
\usepackage{epsfig}
\usepackage{algorithm}
\usepackage{xcolor}
\usepackage{float}
\usepackage{epstopdf}
\usepackage{adjustbox}
\usepackage{subcaption}
\usepackage[numbers]{natbib}
\usepackage{booktabs}
\definecolor{lightblue}{RGB}{0,102,204}
\usepackage[colorlinks=true, linkcolor=lightblue, citecolor=red]{hyperref}
\usepackage{cleveref}
\creflabelformat{equation}{(#2#1#3)}
\crefrangeformat{equation}{(#3#1#4)--(#5#2#6)}
\crefname{equation}{}{}
\usepackage[page]{appendix}

\usepackage[labelfont=bf]{caption} 
\newtheorem{assumption}{Assumption}[section]
\newtheorem{theorem}{Theorem}[section]
\newtheorem{lemma}{Lemma}[section]
\newtheorem{definition}{Definition}[section]
\newtheorem{proposition}{Proposition}[section]

\theoremstyle{definition}

\begin{document}

\title{Inverse Optimal Control for Linear Quadratic Problem with Poisson Jumps: Model-Free Inverse Reinforcement Learning Approaches}
\author{Wen Du,\thanks{School of Statistics and Mathematics, Shandong University of Finance and Economics, Jinan 250014, China. Email:  \url{wwen_du@163.com}.} \and
Na Li,\thanks{Corresponding author. School of Mathematical Sciences, Dalian University of Technology, Dalian 116024, China. Email:  \url{lina2025@dlut.edu.cn}. The author acknowledges financial support from the NSFC (No. 12571475, No. 12171279).} \and
Xun Li,\thanks{Department of Applied Mathematics, The Hong Kong Polytechnic University, Hong Kong SAR, China. Email:  \url{li.xun@polyu.edu.hk}. The author acknowledges financial support from the Research Grants Council of Hong Kong under grant (No. 15225124) and PolyU 4-ZZVB.} \and
Zuo Quan Xu\thanks{Department of Applied Mathematics, The Hong Kong Polytechnic University, Kowloon, Hong Kong SAR, China. Email:  \url{maxu@polyu.edu.hk}. 
The author acknowledges financial support from the NSFC (No. 12571517), Hong Kong RGC (GRF 15203423 and 15204622), The PolyU-SDU Joint Research Center on Financial Mathematics, The CAS AMSS-PolyU Joint Laboratory of Applied Mathematics, The Research Centre for Quantitative Finance (1-CE03), The Hong Kong Polytechnic University.}}

\maketitle

\begin{abstract}
This paper addresses the inverse optimal control (IOC) problem for stochastic linear systems subject to both Brownian motion and Poisson jumps, using an inverse reinforcement learning (IRL) framework. Given a target feedback gain from an expert, the objective is to identify an equivalent cost functional --- specifically, the set of all cost weights --- that yields this same gain. To solve this problem when system dynamics are unknown, we propose two model-free, off-policy IRL algorithms that operate entirely from data, circumventing the need to solve the generalized algebraic Riccati equation or compute the cost weights analytically. The first is an inverse Q-learning algorithm that constructs data-driven equations from expert demonstrations to compute the Q-function matrix, with equivalent cost weights updated algebraically and without requiring additional trajectory data. The second is a model-free off-policy inverse policy iteration algorithm that leverages data collected under an initial stabilizing policy, offering a complementary approach suited to different data availability scenarios. Crucially, by decoupling the data-collection behavior policies from the policies being iteratively updated, both algorithms can learn equivalent cost weights from sufficiently excited trajectories without identifying the system dynamics or jump intensity. Numerical simulations validate the effectiveness of the proposed methods.\\ \par\
\end{abstract}

\newpage

\section{Introduction} \label{section1}
~~~~Many practical problems are influenced by both continuous fluctuations and abrupt variations, such as policy shocks, cyberattacks, and credit defaults. These discontinuous events occur at random times, and cannot be adequately described by diffusion models with continuous sample paths. Therefore, stochastic systems driven by Brownian motion and Poisson jumps provide an important framework for modeling both continuous disturbances and sudden shocks. Given a prescribed cost functional, stochastic optimal control (SOC) seeks to determine the optimal control law. {\O}ksendal and Sulem \cite{oksendal2010maximum} studied SOC problems for forward-backward stochastic differential equations with jumps and established sufficient and necessary maximum principles. Song {\em et al.} \cite{song2020maximum} derived a rigorous stochastic maximum principle for stochastic systems with jumps by introducing a new spike variation technique that handles the estimation difficulties caused by jump terms. In the linear quadratic (LQ) case, Li {\em et al.} \cite{Li-Wu} addressed the indefinite stochastic LQ problem with Poisson jumps through a relaxed compensator, and obtained well-posedness and solvability results for the associated Hamiltonian system and Riccati equation with jumps. Wu {\em et al.} \cite{Wu-Tang} further established the equivalence among open-loop solvability, closed-loop solvability, and the existence of a stabilizing solution to the generalized algebraic Riccati equation in an infinite-horizon setting. However, in many practical problems, the cost functional is unknown, while a target feedback gain or expert trajectories may be available. This leads to inverse optimal control (IOC), which aims to find a cost functional under which the given policy or behavior can be interpreted as optimal. Traditional IOC methods usually rely on known system dynamics and learn cost weights by applying optimality tools such as Riccati equations, linear matrix inequalities, and Karush-Kuhn-Tucker conditions \cite{jean2018inverse,priess2014solutions, ab2020inverse}. Do \cite{do2019inverse} extended IOC to jump-diffusion systems and derived inverse optimal stabilizers without explicitly solving Hamilton-Jacobi-Bellman equations. Recently, Ren {\em et al.} \cite{ren2026inverse} studied inverse optimal incremental control for nonlinear jump-diffusion systems by learning meaningful cost functionals that ensure optimality and incremental stability. Most existing approaches still require prior knowledge of the system dynamics, which becomes restrictive when the system is unknown or difficult to model accurately.

Reinforcement learning (RL) \cite{Sutton-2018} has provided a series of data-driven frameworks for SOC that reduce dependence on exact system dynamics. In particular, RL methods for continuous-time SOC have attracted increasing attention. Jia and Zhou \cite{jia2022policy} developed policy gradient and actor-critic learning algorithms based on martingale characterizations, and subsequently proposed a q-learning framework in which the conventional Q-function is replaced by a first-order q-function \cite{jia2023q}. Li {\em et al.} \cite{Li-Li} introduced an online partially model-free RL approach to compute the optimal control policy without solving the associated stochastic algebraic Riccati equation. Zhao {\em et al.} \cite{Zhao2024} proposed a model-free RL paradigm that simultaneously determines the optimal control gain and optimal dynamic compensator with unknown dynamics and unmeasurable disturbances. For jump-diffusion systems, Guo \cite{guo2023reinforcement} presented a greedy least-squares algorithm that estimates unknown system dynamics online while learning optimal feedback policies, and this algorithm guarantees Lipschitz stability and sublinear regret. Gao {\em et al.} \cite{gao2024reinforcement} designed a q-learning framework for jump-diffusion models and studied the effect of jumps on continuous-time RL algorithms with applications to financial problems. Other recent related works include Zhang and Li \cite{zhang-2024}, Zhang and Jia \cite{Zhang2025}, among others. 

Inverse RL (IRL), introduced by Ng and Russell \cite{Ng-2000}, provides a data-driven approach to address IOC problems. Its objective is to find an equivalent cost functional from observed trajectories that yields the target control gain. IRL algorithms have attracted considerable attention for deterministic systems. Self {\em et al.} \cite{self2022model} developed an online IRL algorithm that simultaneously estimates system parameters and cost weights from input-output data. This algorithm relaxes the requirement for known parameters and enables online adaptation from measured data, while remaining constrained by a predefined system structure. Xue {\em et al.} \cite{Xue-2021inverse} proposed a model-free inverse Q-learning algorithm for discrete-time systems. Their method learns equivalent cost weights using only input-state trajectories, without requiring system dynamics information. Lian {\em et al.} \cite{lian2024inverse} designed inverse value iteration and model-free inverse Q-learning algorithms for continuous-time linear systems. The inverse Q-learning approach eliminates the need for an initial stabilizing control policy. More recently, IRL studies have been extended to stochastic systems. Sun and Jia \cite{Sun-2025} presented a model-free off-policy IRL method for continuous-time stochastic LQ systems, in which equivalent cost weights are learned without prior knowing the system dynamics. However, IRL for jump-diffusion systems remains largely unexplored. This motivates the development of data-driven model-free IRL methods for IOC problems with Poisson jumps. 

In this paper, we study an inverse stochastic LQ (ISLQ) problem and develop a two-agent IRL framework consisting of a target expert agent and a learner agent. Both agents share the same linear system dynamics. The learner seeks an equivalent cost functional that yields the same feedback gain as the expert, without prior knowledge of the system parameters or jump intensity. To solve the ISLQ problem in a data-driven manner, we propose two model-free off-policy IRL algorithms based on different data sources. The first is an off-policy inverse Q-learning algorithm, which constructs data-driven Q-function equations directly from expert demonstrations. The second is a model-free off-policy inverse policy iteration algorithm, which uses trajectories generated by the learner under an initial stabilizing behavior policy. In both algorithms, the behavior policy used for data collection is decoupled from the policy updated during iteration. The main contributions of this study are listed as follows.
\begin{enumerate}[1)]
\item In contrast to the existing IOC method for stochastic systems with Poisson jumps \cite{do2019inverse}, this paper proposes two completely model-free IRL algorithms. The proposed methods learn equivalent cost weights using either the expert's demonstrations or the learner's trajectories collected over local time intervals, without identifying the system dynamics or jump intensity. The convergence and stability of the proposed algorithms are rigorously proved.
\item The inverse Q-learning algorithm is devoted to solving the IOC problem using the expert's demonstrated trajectories. The algorithm constructs data-driven equations in terms of the Q-function and uses the expert trajectories to calculate the Q-function matrix, based on which the equivalent cost weight is updated algebraically without requiring additional trajectory data. A probing noise is introduced into the expert input to guarantee the persistent excitation rank condition and enrich the data.
\item The model-free off-policy inverse policy iteration algorithm is proposed as a counterpart to the inverse Q-learning algorithm. In this algorithm, the learner collects data under an initial stabilizing control policy and uses the collected trajectories to iteratively update the equivalent cost weights. This provides an alternative data-driven framework when the learner's trajectories are available. Sufficient excitation is added to the behavior policy to ensure the solvability of the learning equations at each iteration.
\end{enumerate}

The remainder of this paper is organized as follows. Section \hyperref[section2]{2} formulates an ISLQ problem for jump-diffusion systems and introduces the definition of equivalent cost weights. Section \hyperref[section3]{3} develops an off-policy inverse Q-learning algorithm using the expert's demonstrated trajectories, where the Q-function formulation is employed to obtain data-driven equations without using system parameters. Section \hyperref[section4]{4} proposes a model-free off-policy inverse policy iteration algorithm based on the learner's trajectories collected under an initial stabilizing behavior policy. The stability and convergence of the proposed algorithms are established in the corresponding sections. Section \hyperref[section5]{5} provides a numerical example to illustrate their effectiveness. Finally, Section \hyperref[section6]{6} concludes the paper.

\textit{Notation}: Let $(\Omega,\mathcal{F},\mathbb{P},\mathbb{F})$ be a complete filtered probability space. The filtration $\mathbb{F}=\left\{{\mathcal{F}_t}\right\}_{t\geq 0}$ is generated by two mutually independent stochastic processes and augmented the collection of all P-null sets. One is a standard one-dimensional Brownian motion $\{W(t)\}_{t \geq 0}$, 
and the other is a Poisson random measure $\{N(\cdot,\cdot)\}$ defined on $\mathbb{R}_+\times \mathcal{E}$, where $\mathcal{E}=\mathbb{R}\setminus \{0\}$ is a nonempty Borel subset of some Euclidean space. The compensator of $N(\cdot,\cdot)$ is $\overline{N}(dt,de)=\lambda (de)dt$ which make $\left\{\widetilde{N}((0,t]\times S)=(N-\overline{N})((0,t]\times S); ~ 0 \leq t < \infty\right\}$ a martingale for any $S\in\mathcal{B}(\mathcal{E})$ with $\lambda(\mathcal{E})<\infty$. Here, $\lambda$ is a given $\sigma$-finite measure on the measurable space $(\mathscr{E},\mathscr{B}(\mathscr{E}))$ such that $\int _\mathscr{E}(1\wedge e^2 )\lambda(de)<\infty$. 
Denote by $\mathbb R^n$ the $n$-dimensional Euclidean space and by $\mathbb{R}^{n\times m}$ the set of all $n\times m$ real matrices. Let $\langle \cdot, \cdot \rangle $ be the inner product on $\mathbb{R}^n$. For a given vector or matrix $A$, $A^{\top}$ denotes its transpose. The symbol ${\bf{0}}$ denotes a zero vector or matrix with appropriate dimension and $\varnothing$ denotes the empty set. Furthermore, $\mathbb{E}^{\mathcal{F}_t} =\mathbb{E}\left[ \cdot|\mathcal{F}_t \right]$ stands for the conditional expectation operator. We use $\mathbb{S}^n$, $\mathbb{S}^n_+$ and $\mathbb{S}^n_{++}$ to denote the set of all symmetric matrices, non-negative definite matrices and positive definite matrices in $\mathbb{R}^{n\times n}$. As usual, if a matrix $A \in \mathbb{S}_{+}^n$ (respectively, $\mathbb{S}^n_{++}$), we write $A \geq \mathbf{0} $ (respectively, $> \mathbf{0}$). For matrices $A,B\in \mathbb{S}^{n}$, we write $A\geq B$ (respectively, $A > B$) if $A - B \geq \mathbf{0} $ (respectively, $> \mathbf{0}$). For a Euclidean space \(\mathbb H\) with norm $\left\|\cdot\right\|_{\mathbb H}$, we define the Hilbert space $L^2_{\mathbb{F}}(\mathbb H)$, which is the space of $\mathbb H$-valued and $\mathbb{F}$-progressively measurable processes $\left\{f(t,w), (t,w)\in \left[0,\infty\right) \times \Omega \right\}$ such that $\mathbb{E}\int^{\infty}_0 \left\|f(t,w)\right\|_{\mathbb H}^2dt<\infty$. Moreover, \(L^{\lambda,2}(\mathcal E;\mathbb H)\) denotes the space of all $\mathbb H$-valued measurable functions $\{r(e),e \in \mathcal{E}\}$ defined on the measurable space $\left(\mathcal{E},\mathcal{B}(\mathcal{E});\lambda\right)$ satisfying $\int _{\mathcal{E}} \left\|r(e)\right\|_{\mathbb H} ^2\lambda (de)<\infty$. For any matrix $P\in \mathbb{S}^n$, we define the following vectorization operators:
$$
vec(P)=[p_{11},p_{21}, \dots, p_{n1},p_{12}, p_{22},\dots, p_{n-1,n}, p_{nn}]^{\top},
$$
$$
vec^{+}(P)=[p_{11},2p_{12}, \dots ,2p_{1n},p_{22},2p_{23},\dots ,2p_{n-1,n},p_{nn}]^{\top},
$$
where $vec(P) \in \mathbb{R}^{n^2}$ and $vec^{+}(P) \in \mathbb{R}^{{\frac{1}{2}}n(n+1)}$. Let $\otimes$ be the Kronecker product. If $A$, $B$ and $C$ have appropriate dimensions, then $vec(ABC)=\left(C^{\top}\otimes A\right)vec(B)$.

\section{Preliminaries and Problem Formulation} \label{section2}
~~~~In this section, we establish a two-agent IRL framework consisting of a target expert agent and a learner agent sharing the same system dynamics. The learner seeks to find an equivalent cost functional whose optimal feedback gain coincides with the target gain. Based on this framework, we introduce the notion of equivalent cost weights, formulate the ISLQ problem, and characterize its solution set.

We consider the following stochastic linear dynamics driven by Brownian motion and Poisson jumps:
\begin{equation}\label{expert-system}
\left\{
\begin{aligned}
dX(s)&=[AX(s)+Bu(s)]ds+[CX(s)+Du(s)]dW(s)\\
&~~~+\int_{\mathcal{E}}[E(e)X(s-)+F(e)u(s)]\widetilde{N}(ds,de),~~s\geq t, \\
X(t)&=x,
\end{aligned}
\right.
\end{equation}
where the coefficients $A,C\in \mathbb R^{n\times
n}$, $B,D \in \mathbb R^{n\times
m}$ are constant matrices, $E(\cdot) \in L^{\lambda,2}(\mathcal{E};\mathbb R^{n\times n})$ and $F(\cdot) \in
L^{\lambda,2}(\mathcal{E};\mathbb R^{n\times m})$ are given deterministic matrix-valued functions. In this system, the process $X(\cdot)\in L^2_{\mathbb{F}}(\mathbb R^n)$ is the state process, $u(\cdot)\in L^2_{\mathbb{F}}(\mathbb R^m) $ is the control process, and $x \in \mathbb{R}^n$ is the initial state at the initial time $t$.

For any given initial state $x \in \mathbb{R}^n$ and initial time $t\geq0$, the expert agent is associated with the following quadratic target cost functional
\begin{equation} \label{target cost functional}
\begin{aligned}
\mathcal{J}_{\mathcal{T}}(t,x;u(\cdot)) :=\mathbb{E}^{\mathcal{F}_t}
\int_t^{\infty}\left[\langle
N_\mathcal{T}X(s),X(s)\rangle+\langle
R_\mathcal{T}u(s),u(s)\rangle\right]ds,
\end{aligned}
\end{equation}
where $N_\mathcal{T}\in \mathbb{S}_{++}^{m}$ and $R_\mathcal{T}\in \mathbb{S}_{++}^{m}$ are the target state-penalty and input-penalty weights, respectively, with appropriate dimensions. 

To ensure that the above infinite-horizon cost functional is well-defined, we introduce the concept of $L^{2}$-stabilizability.

\begin{definition} \label{definition2.1}
{\rm 
System \eqref{expert-system} is said to be $L^{2}$-stabilizable if there exists a constant matrix $K \in \mathbb R^{m \times n}$ such that the process $X$ driven by 
\begin{equation} \label{feedback SDE}
\left\{ 
\begin{aligned}
dX(s)&=\left(A+BK\right)X(s)ds+\left(C+DK\right)X(s)dW(s)\\
&~~~+\int_{\mathcal{E}}\left(E(e)+F(e)K\right)X(s)\widetilde{N}(ds,de),~~s\geq t, \\
X(t)&=x,
\end{aligned}
\right.
\end{equation}
satisfies $\lim_{s \to \infty}$ $\mathbb{E} [{X(s)}^{\top}X(s)]=0$. In this case, the closed-loop system \eqref{feedback SDE} is called $L^{2}$-stable. The feedback control $u(\cdot)=KX(\cdot)$ is called a stabilizing control for system \eqref{expert-system} with the matrix $K$ being a stabilizer.}
\end{definition}

\begin{assumption} \label{assumption1}
\rm{System \eqref{expert-system} is $L^{2}$-stabilizable.}
\end{assumption}

\begin{assumption} \label{assumption3}
\rm{All cost weight matrices associated with the quadratic cost functionals considered throughout this paper are positive definite.}
\end{assumption}

The following lemma provides an equivalent condition for the existence of stabilizers for system \eqref{expert-system}.

\begin{lemma}\cite[Theorem~4.1]{Wu-Tang} \label{lemma 2.3}
{\rm A matrix $K\in \mathbb{R}^{m\times n}$ is a stabilizer of system \eqref{expert-system} if and only if there exists a matrix $P\in \mathbb{S}_{++}^n$ such that }
\[
\begin{aligned}
(A+BK)^{\top}P&+P(A+BK)+(C+DK)^{\top}P(C+DK)\\
&+\int_{\mathcal{E}}(E(e)+F(e)K)^{\top}P(E(e)+F(e)K)\lambda(de) < \mathbf{0}.
\end{aligned}
\]
{\rm In this case, the following Lyapunov equation 
\begin{align*}
(A+BK)^{\top}P+P(A+BK)&+(C+DK)^{\top}P(C+DK)\\
&+\int_{\mathcal{E}}(E(e)+F(e)K)^{\top}P(E(e)+F(e)K)\lambda(de)+\Xi=\mathbf{0} 
\end{align*}
admits a unique solution $P \in \mathbb{S}^n$ (respectively, $\mathbb{S}_{+}^n$ and $\mathbb{S}^n_{++}$) for any $\Xi \in \mathbb{S}^{n}$ (respectively, $\mathbb{S}_{+}^n$ and $\mathbb{S}^n_{++}$).
}
\end{lemma}

Under Assumption \hyperref[assumption1]{2.1}, we define the set of admissible controls as 
\begin{equation*}
\mathcal{U}_{ad} :=\big\{u(\cdot)\in L^2_{\mathbb{F}}(\mathbb{R}^m)\;\big|\; u(\cdot)~\text{is a stabilizing control}\big\}.
\end{equation*}

The expert seeks an admissible control $u_\mathcal{T}(\cdot) \in \mathcal{U}_{ad}$ such that
\begin{equation} \label{eq3}
\begin{aligned}
\mathcal{J}_{\mathcal{T}}(t,x;u_\mathcal{T}(\cdot))= \inf_{u(\cdot)\in\mathcal{U}_{ad}}\mathcal{J}_{\mathcal{T}}(t,x;u(\cdot))\triangleq V_\mathcal{T}(t,x),
\end{aligned}
\end{equation}
where $V_\mathcal{T}(t,x)$ is called the target value function and $u_\mathcal{T}(\cdot)$ is called the target optimal control. Correspondingly, the solution $X_\mathcal{T}(\cdot)$ to \eqref{expert-system} under $u_\mathcal{T}(\cdot)$ is called the target optimal trajectory. Based on \cite[Theorem~5.4]{Wu-Tang}, $u_\mathcal{T}$ has the following feedback form 
\begin{equation} \label{eq5}
\begin{aligned}
u_\mathcal{T}(s)
=K_\mathcal{T}X_\mathcal{T}(s)
\end{aligned}
\end{equation}
with the target control gain
\begin{align} 
K_\mathcal{T}=&-\left(R_\mathcal{T}+D^{\top}P_\mathcal{T}D+\int_{\mathcal{E}}{F(e)}^{\top}P_\mathcal{T}F(e)\lambda(de)\right)^{-1}\left(B^{\top}P_\mathcal{T}+D^{\top}P_\mathcal{T}C+\int_{\mathcal{E}}{F(e)}^{\top}P_\mathcal{T}E(e)\lambda(de)\right), \label{expert gain}
\end{align}
and $P_\mathcal{T} \in \mathbb{S}_{++}^{n}$ is the solution of the following stochastic Riccati equation with Poisson jumps (SREP)
\begin{equation} \label{eq4}
\begin{aligned}
&P_\mathcal{T}A+A^{\top}P_\mathcal{T}+C^{\top}P_\mathcal{T}C+\int_{\mathcal{E}}E(e)^{\top}P_\mathcal{T}E(e)\lambda(de)+N_\mathcal{T}-\left(P_\mathcal{T}B+C^{\top}P_\mathcal{T}D+\int_{\mathcal{E}}E(e)^{\top}P_\mathcal{T}F(e)\lambda(de)\right)\\
&\left(R_\mathcal{T}+D^{\top}P_\mathcal{T}D+\int_{\mathcal{E}}F(e)^{\top}P_\mathcal{T}F(e)\lambda(de)\right)^{-1}\left(B^{\top}P_\mathcal{T}+D^{\top}P_\mathcal{T}C+\int_{\mathcal{E}}F(e)^{\top}P_\mathcal{T}E(e)\lambda(de)\right)=\mathbf{0}.
\end{aligned}
\end{equation} 

For a given $R \in\mathbb{S}_{++}^m$ , which may differ from $R_\mathcal{T}$ given in \eqref{target cost functional}, the learner aims to find the following equivalent cost functional.

\begin{definition} {\rm (Equivalent cost functional)}\label{definition2.2} 
{\rm A cost functional 
\begin{align}
\mathcal{J}(t,x;u(\cdot)) :=\mathbb{E}^{\mathcal{F}_t}
\int_t^{\infty}\left[\langle
NX(s),X(s)\rangle+\langle
Ru(s),u(s)\rangle\right]ds \label{learner function}
\end{align} 
with $ N \in \mathbb{S}^n_{++}$ and $R \in \mathbb{S}^m_{++}$ is called an equivalent cost functional to \eqref{target cost functional} subject to system \eqref{expert-system} if its optimal feedback gain coincides with the target gain $K_\mathcal{T}$. In this case, $\left(N,R\right)$ is called an equivalent weight pair to $\left(N_\mathcal{T},R_\mathcal{T}\right)$. If $R \in \mathbb{S}^m_{++}$ is given, then $N$ is called an equivalent weight to $N_\mathcal{T}$.}
\end{definition}

We now make the following assumption that describes the information available to the learner and formulate the ISLQ problem.

\begin{assumption} \label{assumption2}
{\rm 1) The expert's target gain $K_\mathcal{T}$ is available.
2) The system parameters in \eqref{expert-system}, the jump intensity, and target cost weights $N_\mathcal{T}$ and $R_\mathcal{T}$ in \eqref{target cost functional} are unknown.
}
\end{assumption}

\noindent\textbf{Problem (ISLQ):}
Under Assumptions \hyperref[assumption3]{2.2} and \hyperref[assumption2]{2.3}, for a given $R \in\mathbb{S}_{++}^{m}$, the learner's objective is to find an equivalent weight $N\in \mathbb{S}_{++}^{n}$ corresponding to $N_\mathcal{T}$.

The following theorem establishes a sufficient condition for the solution to Problem (ISLQ).

\begin{theorem}\label{theorem2.1}
{\rm Given $R \in \mathbb{S}_{++}^{m}$. Suppose that there exist $P$, $N \in \mathbb{S}_{++}^{n}$ satisfy the following SREP
\begin{align}
&PA+A^{\top}P+C^{\top}PC+\int_{\mathcal{E}}E(e)^{\top}PE(e)\lambda(de)+N-\left(PB+C^{\top}PD+\int_{\mathcal{E}}E(e)^{\top}PF(e)\lambda(de)\right) \nonumber\\
&\left(R+D^{\top}PD+\int_{\mathcal{E}}F(e)^{\top}PF(e)\lambda(de)\right)^{-1}\left(B^{\top}P+D^{\top}PC+\int_{\mathcal{E}}F(e)^{\top}PE(e)\lambda(de)\right)=\mathbf{0} \label{eq8}
\end{align} 
and Lyapunov equation}
\begin{align}
P(A+BK_\mathcal{T})&+(A+BK_\mathcal{T})^{\top}P+(C+DK_\mathcal{T})^{\top}P(C+DK_\mathcal{T})\nonumber \\
&+\int_{\mathcal{E}}(E(e)+F(e)K_\mathcal{T})^{\top}P(E(e)+F(e)K_\mathcal{T})\lambda(de)+N+K_\mathcal{T}^{\top}RK_\mathcal{T}=\mathbf{0}, \label{eq10}
\end{align}
{\rm then $N$ is an equivalent weight to $N_\mathcal{T}$.}
\end{theorem}

\begin{proof} 
According to \cite[Theorem~5.4]{Wu-Tang}, the optimal control gain $K$ associated with SREP \eqref{eq8} is given by
\begin{align}
K=&-\left(R+D^{\top}PD+\int_{\mathcal{E}}{F(e)}^{\top}PF(e)\lambda(de)\right)^{-1}\left(B^{\top}P+D^{\top}PC+\int_{\mathcal{E}}{F(e)}^{\top}PE(e)\lambda(de)\right). \label{eq9}
\end{align}
Rewrite SREP \eqref{eq8} using $K_\mathcal{T}$ in \eqref{expert gain} and $K$ in \eqref{eq9} as 
\begin{align}
&P(A+BK_\mathcal{T})+(A+BK_\mathcal{T})^{\top}P+(C+DK_\mathcal{T})^{\top}P(C+DK_\mathcal{T}) \nonumber\\
&+\int_{\mathcal{E}}(E(e)+F(e)K_\mathcal{T})^{\top}P(E(e)+F(e)K_\mathcal{T})\lambda(de)+N \nonumber \\
&+K_\mathcal{T}^{\top}(R+D^{\top}PD+\int_{\mathcal{E}}F(e)^{\top}PF(e)\lambda(de))K \nonumber \\
&+K^{\top}(R+D^{\top}PD+\int_{\mathcal{E}}F(e)^{\top}PF(e)\lambda(de))K_\mathcal{T} \nonumber \\
&-K_\mathcal{T}^{\top}(D^{\top}PD+\int_{\mathcal{E}}F(e)^{\top}PF(e)\lambda(de))K_\mathcal{T} \nonumber \\
&-K^{\top}(R+D^{\top}PD+\int_{\mathcal{E}}F(e)^{\top}PF(e)\lambda(de))K=0. \label{star}
\end{align}
To establish the relationship between $K$ and $K_\mathcal{T}$, we subtract \eqref{eq10} from \eqref{star}, obtaining 
\begin{equation} \label{star2}
\begin{aligned}
\left(K_\mathcal{T}-K\right)^{\top}\left(R+D^{\top}PD+\int_{\mathcal{E}}F(e)^{\top}PF(e)\lambda(de)\right)(K_\mathcal{T}-K)=\mathbf{0}.
\end{aligned}
\end{equation}

Since $R \in \mathbb{S}_{++}^{m}$ and $P \in \mathbb{S}_{++}^{n}$, \eqref{star2} implies $K=K_\mathcal{T}$. By Definition \ref{definition2.2}, $N$ is an equivalent weight to $N_\mathcal{T}$. This completes the proof.
\end{proof}

The above analysis shows that different cost weights may yield the same target feedback gain $K_\mathcal{T}$, and hence the equivalent cost functional to \eqref{target cost functional} subject to \eqref{expert-system} may not be unique. The following theorem characterizes the solution set of Problem (ISLQ). 

\begin{theorem}[\rm Equivalence of cost functionals]\label{theorem2.2}
{\rm Suppose that $P_\mathcal{T}\in \mathbb{S}_{++}^{n}$ is the unique solution to \eqref{eq4}. For given $R \in \mathbb{S}_{++}^{m}$ and $R_o=R_\mathcal{T}-R \in \mathbb{S}_{++}^{m}$, if there exist matrices $P_o \in \mathbb{S}_{++}^{n}$ and $N_o \in \mathbb{S}_{++}^{n}$ satisfying}
\begin{align}
B^{\top}P_{o}+D^{\top}P_{o}C+\int_{\mathcal{E}}F(e)^{\top}P_{o}E(e)\lambda(de)=-\left(R_{o}+D^{\top}P_{o}D+\int_{\mathcal{E}}F(e)^{\top}P_{o}F(e)\lambda(de)\right)K_\mathcal{T}, \label{eq37}
\end{align}
{\rm and}
\begin{align}
&P_{o}A+A^{\top}P_{o}+C^{\top}P_{o}C+\int_{\mathcal{E}}E(e)^{\top}P_{o}E(e)\lambda(de)+N_{o} \nonumber \\
&-{K_\mathcal{T}}^{\top}\left(R_{o}+D^{\top}P_{o}D+\int_{\mathcal{E}}F(e)^{\top}P_{o}F(e)\lambda(de)\right)K_\mathcal{T}=\mathbf{0}, \label{eq38}
\end{align}
{\rm then $\mathcal{J}_{o}(t,x;u(\cdot))$ in the form of \eqref{learner function} with $R_o$ and $N_o$ is equivalent to $\mathcal{J}_\mathcal{T}(t,x;u(\cdot))$. Moreover, $\mathcal{J}(t,x;u(\cdot))$ is also equivalent to $\mathcal{J}_\mathcal{T}(t,x;u(\cdot))$ and}
\begin{equation*}
\mathcal{J}_\mathcal{T}(t,x;u(\cdot))=\mathcal{J}(t,x;u(\cdot))+\mathcal{J}_{o}(t,x;u(\cdot)).
\end{equation*}
\end{theorem}

\begin{proof}
Together with system \eqref{expert-system} and cost functional $\mathcal{J}_{o}(t,x;u(\cdot))$, we can formulate an auxiliary SLQ problem. By \eqref{eq37} and \eqref{eq38}, $P_{o}$ satisfies the following SREP
\begin{align}
    &P_{o}A+A^{\top}P_{o}+C^{\top}P_{o}C+\int_{\mathcal{E}}E(e)^{\top}P_{o}E(e)\lambda(de)+N_{o} \nonumber \\
    &-{K_o}^{\top}\left(R_{o}+D^{\top}P_{o}D+\int_{\mathcal{E}}F(e)^{\top}P_{o}F(e)\lambda(de)\right)K_o=\mathbf{0}. \label{eq38-a}
\end{align}
In this optimal problem, by \cite[Theorem~5.4]{Wu-Tang}, the corresponding optimal feedback gain $K_{o}$ is given by
\begin{equation*}
    \begin{aligned}
        K_{o}=&-\left(R_{o}+D^{\top}P_{o}D+\int_{\mathcal{E}}F(e)^{\top}P_{o}F(e)\lambda(de)\right)^{-1} \\ 
        &\left(B^{\top}P_{o}+D^{\top}P_{o}C+\int_{\mathcal{E}}F(e)^{\top}P_{o}E(e)\lambda(de)\right).
    \end{aligned}
\end{equation*}
Comparing to \eqref{eq37}, we have $K_{o}=K_\mathcal{T}$, which implies that $\mathcal{J}_{o}(t,x;u(\cdot))$ is equivalent to $\mathcal{J}_\mathcal{T}(t,x;u(\cdot))$.

Denote $P=P_\mathcal{T}-P_o$ and $N=N_\mathcal{T}-N_o$. Substituting $N_\mathcal{T} = N + N_o$, $P_\mathcal{T} = P + P_o$, and $R_\mathcal{T}=R+R_o$ into \eqref{eq4} gives 
\begin{align*}
&PA+P_{o}A+A^{\top}P+A^{\top}P_{o}+C^{\top}PC+C^{\top}P_{o}C\\
&+\int_{\mathcal{E}}E(e)^{\top}PE(e)\lambda(de)+\int_{\mathcal{E}}E(e)^{\top}P_{o}E(e)\lambda(de)+N+N_{o}\\
&-{K_\mathcal{T}}^{\top}\left(R+D^{\top}PD+\int_{\mathcal{E}}F(e)^{\top}PF(e)\lambda(de)\right)K_\mathcal{T}\\
&-{K_\mathcal{T}}^{\top}\left(R_{o}+D^{\top}P_{o}D+\int_{\mathcal{E}}F(e)^{\top}P_{o}F(e)\lambda(de)\right)K_\mathcal{T}=\mathbf{0}.
\end{align*}
Subtracting \eqref{eq38} from the above equation, we have
\begin{align}
P&A+A^{\top}P+C^{\top}PC+\int_{\mathcal{E}}E(e)^{\top}PE(e)\lambda(de)+N \nonumber\\
&-{K_\mathcal{T}}^{\top}(R+D^{\top}PD+\int_{\mathcal{E}}F(e)^{\top}PF(e)\lambda(de))K_\mathcal{T}=\mathbf{0}. \label{eq39}
\end{align}
Substituting $R_o=R_\mathcal{T}-R$ and $P_o=P_\mathcal{T}-P$ into \eqref{eq37} yields
\begin{align*}
&B^{\top}P+D^{\top}PC+\int_{\mathcal{E}}F(e)^{\top}PE(e)\lambda(de) \nonumber \\
=&-\left(R_\mathcal{T}+D^{\top}P_\mathcal{T}D+\int_{\mathcal{E}}F(e)^{\top}P_\mathcal{T}F(e)\lambda(de)\right)K_\mathcal{T} \nonumber \\
&+\left(R_{o}+D^{\top}P_{o}D+\int_{\mathcal{E}}F(e)^{\top}P_{o}F(e)\lambda(de)\right)K_\mathcal{T} \nonumber \\
=&\left(R+D^{\top}PD+\int_{\mathcal{E}}F(e)^{\top}PF(e)\lambda(de)\right)K_\mathcal{T},  
\end{align*}
then 
\begin{align}
    K_\mathcal{T}=&-\left(R+D^{\top}PD+\int_{\mathcal{E}}F(e)^{\top}PF(e)\lambda(de)\right)^{-1} \nonumber \\
    &\left(B^{\top}P+D^{\top}PC+\int_{\mathcal{E}}F(e)^{\top}PE(e)\lambda(de)\right). \label{eq*}
\end{align}
Inserting \eqref{eq*} into \eqref{eq39}, thus $P$ satisfies \eqref{eq8}. From \eqref{eq*} and \eqref{eq9}, we obtain $K=K_\mathcal{T}$. Therefore, $\mathcal{J}(t,x;u(\cdot))$ is also equivalent to $\mathcal{J}_\mathcal{T}(t,x;u(\cdot))$.

Substituting $N=N_\mathcal{T}-N_{o}$ and $R=R_\mathcal{T}-R_{o}$ into \eqref{target cost functional}, it holds that
\begin{align*}
\mathcal{J}_\mathcal{T}(t,x;u(\cdot))=\mathcal{J}(t,x;u(\cdot))+\mathcal{J}_{o}(t,x;u(\cdot)).
\end{align*}
This completes the proof.
\end{proof}

According to Theorem \hyperref[theorem2.2]{2.2}, $N$ is not unique and target cost functional $\mathcal{J}_\mathcal{T}(t,x;u(\cdot))$ can be represented as the sum of two equivalent cost functionals.

\section{Inverse Q-Learning Algorithm for Problem (ISLQ)} \label{section3}
~~~~In this section, we develop an off-policy inverse Q-learning algorithm for solving Problem (ISLQ), where the expert's demonstrated trajectories $X_\mathcal{T}(\cdot)$ and the target control gain $K_{\mathcal{T}}$ are available. By formulating the inverse problem through the Q-function, the proposed method is implemented using these demonstrations, without requiring any information about the system dynamics or jump intensity. This algorithm is off-policy in the sense that the behavior policy used for data collection is decoupled from the feedback gains updated during iteration. The convergence and stability of the proposed algorithm are then rigorously established.

\subsection{Stability and Convergence of an Inverse Q-Learning Algorithm}
~~~~First, we define a $Q$-function as follows.
\begin{align}
Q\left(x,u\right) :=&~\left[\begin{array}{c}
x \\
u
\end{array}
\right]^{\top} {\mathbf{Q}} \left[\begin{array}{c}
x \\
u
\end{array}
\right] \nonumber\\
=&~\left(u-\Gamma({\mathbf{Q}})x\right)^{\top}\left(R+D^{\top}PD+\int_{\mathcal{E}}F(e)^{\top}PF(e)\lambda(de)\right)\left(u-\Gamma({\mathbf{Q}})x\right) \nonumber\\
& 
+ {x}^{\top}\Pi({\mathbf{Q}})x , \label{eqC1}
\end{align}
where 
\begin{align*}\mathbf{Q}&=\left[\begin{array}{cc}
Q_{xx}(P) & Q_{xu}(P) \\
Q_{ux}(P) & Q_{uu}(P)
\end{array}
\right],\\
\Pi({\mathbf{Q}})&=Q_{xx}(P)-Q_{xu}(P)\left(Q_{xx}(P)\right)^{-1}Q_{ux}(P),\\
\Gamma({\mathbf{Q}}) &=-\left(Q_{uu}(P)\right)^{-1}Q_{ux}(P)
\end{align*}
with
\begin{align*}
Q_{xx}(P) &= PA+A^{\top}P+C^{\top}PC+\int_{\mathcal{E}}{E(e)}^{\top}PE(e)\lambda(de)+N+P,\\
Q_{ux}(P) &= {Q_{xu}}^{\top} = B^{\top}P+D^{\top}PC+\int_{\mathcal{E}}{F(e)}^{\top}PE(e)\lambda(de),\\
Q_{uu}(P) &= R+D^{\top}PD+\int_{\mathcal{E}}F(e)^{\top}PF(e)\lambda(de). 
\end{align*}
For simplicity, we denote 
\begin{align*}
Q_{xx}^{(i+1)}&=Q_{xx}\left(P^{(i+1)}\right), &
Q_{ux}^{(i+1)}&={Q_{xu}^{(i+1)}}^{\top}=Q_{ux}\left(P^{(i+1)}\right),\\
Q_{uu}^{(i+1)}&=Q_{uu}\left(P^{(i+1)}\right), &
\mathbf{Q}^{(i+1)}&=\left[\begin{array}{cc}
Q_{xx}^{(i+1)} & Q_{xu}^{(i+1)} \\
Q_{ux}^{(i+1)} & Q_{uu}^{(i+1)}
\end{array}
\right].
\end{align*}

Now, we present the inverse Q-learning algorithm. 

\begin{algorithm}[H]
\caption{\rm Inverse Q-learning Algorithm} \label{Algorithm 3} 
1: {\bf Initialization}: For a given $R>{\bf 0}$, choose an initial $N^{(0)}>{\bf 0}$. Let $i=0$ and $\varepsilon>0$.\\
2: {\bf do} $\{$\\
3: \textbf{Q-function Evaluation}: Solve $\mathbf{Q}^{(i+1)}$ by the equation\\
\begin{align}\label{algorithm3-correction}
&\begin{bmatrix}
X_\mathcal{T}(t) \\
u_\mathcal{T}(t)
\end{bmatrix}^{\top} {\mathbf{Q}^{(i+1)}} \begin{bmatrix}
X_\mathcal{T}(t) \\
u_\mathcal{T}(t)
\end{bmatrix}
- \mathbb{E}^{\mathcal{F}_t}\left\{\begin{bmatrix}
X_\mathcal{T}(t+\Delta t) \\
u_\mathcal{T}(t+\Delta t)
\end{bmatrix}^{\top} {\mathbf{Q}^{(i+1)}} \begin{bmatrix}
X_\mathcal{T}(t+\Delta t) \\
u_\mathcal{T}(t+\Delta t)
\end{bmatrix}\right\} \nonumber\\
=&~\mathbb{E}^{\mathcal{F}_t}\int_t^{t+\Delta t}{X_\mathcal{T}(s)}^{\top}\left[N^{(i)}+{K_\mathcal{T}}^{\top}RK_\mathcal{T}\right]X_\mathcal{T}(s)ds.
\end{align}\\
4: \textbf{Policy Improvement}: Update $K^{(i+1)}$ by\\
\begin{align} \label{algorithm3-update}
K^{(i+1)}= &~\Gamma(\mathbf{Q}^{(i+1)})= -\left(Q_{uu}^{(i+1)}\right)^{-1}Q_{ux}^{(i+1)}.
\end{align}
5: \textbf{Weight Update}: Update $N^{(i+1)}$ via the identity\\
\begin{align} \label{algorithm3-reconstruction}
N^{(i+1)}=N^{(i)}+\left(K_\mathcal{T}-K^{(i+1)}\right)^{\top}Q_{uu}^{(i+1)}\left(K_\mathcal{T}-K^{(i+1)}\right).
\end{align}
6: $i\leftarrow i+1$.\\
7: $\}$ {\bf until} $\|N^{(i+1)}-N^{(i)}\|\leq \varepsilon$. 
\end{algorithm}

\begin{proposition} \label{proposition5.3}
{\rm Solving equation \eqref{algorithm3-correction} is equivalent to solving the following Lyapunov recursion:
\begin{align}
P^{(i+1)}&(A+BK_\mathcal{T})+(A+BK_\mathcal{T})^{\top}P^{(i+1)}+(C+DK_\mathcal{T})^{\top}P^{(i+1)}(C+DK_\mathcal{T}) \nonumber\\
&+\int_{\mathcal{E}}(E(e)+F(e)K_\mathcal{T})^{\top}P^{(i+1)}(E(e)+F(e)K_\mathcal{T})\lambda(de)+N^{(i)}+K_\mathcal{T}^{\top}RK_\mathcal{T}=\mathbf{0}. \label{algorithm 1-correction}
\end{align}
}
\end{proposition}

\begin{proof}
Applying $\mathrm{It\hat{o}}$'s formula to ${X_\mathcal{T}(s)}^{\top}P^{(i+1)}X_\mathcal{T}(s)$, we have
\begin{align}
&d\left[{X_\mathcal{T}(s)}^{\top}P^{(i+1)}X_\mathcal{T}(s)\right] \nonumber\\
=&~\left\{{X_\mathcal{T}(s)}^{\top}\left(P^{(i+1)}A+A^{\top}P^{(i+1)}+C^{\top}P^{(i+1)}C+\int_{\mathcal{E}}{E(e)}^{\top}P^{(i+1)}E(e)\lambda(de)\right)X_\mathcal{T}(s)\right. \nonumber\\
&+u_\mathcal{T}(s)^{\top}\left(B^{\top}P^{(i+1)}+D^{\top}P^{(i+1)}C+\int_{\mathcal{E}}{F(e)}^{\top}P^{(i+1)}E(e)\lambda(de)\right)X_\mathcal{T}(s) \nonumber\\
&+X_\mathcal{T}(s)^{\top}\left(P^{(i+1)}B+C^{\top}P^{(i+1)}D+\int_{\mathcal{E}}{E(e)}^{\top}P^{(i+1)}F(e)\lambda(de)\right)u_\mathcal{T}(s) \nonumber\\
&\left.+u_\mathcal{T}(s)^{\top}\left(D^{\top}P^{(i+1)}D+\int_{\mathcal{E}}F(e)^{\top}P^{(i+1)}F(e)\lambda(de)\right)u_\mathcal{T}(s)\right\}ds \nonumber\\
&+\left\{\dots\right\}dW(s)+\left\{\dots\right\}\widetilde{N}(ds,de). \label{eq16}
\end{align}
Integrating \eqref{eq16} from $t$ to $t+\Delta t$ and taking the conditional expectation $\mathbb{E}^{\mathcal{F}_t}$, we deduce that
\begin{align}
&~\mathbb{E}^{\mathcal{F}_t}\left[{X_\mathcal{T}(t+\Delta t)^{\top}}P^{(i+1)}X_\mathcal{T}(t+\Delta t)-X_\mathcal{T}(t)^{\top}P^{(i+1)}X_\mathcal{T}(t)\right] \nonumber\\
= &~\mathbb{E}^{\mathcal{F}_t}\int_t^{t+\Delta t}\left[\begin{array}{c}
X_\mathcal{T}(s) \\
u_\mathcal{T}(s)
\end{array}
\right]^{\top}\left[\begin{array}{cc}
Q^{(i+1)}_{xx}-P^{(i+1)}-N^{(i)} & Q^{(i+1)}_{xu} \\
Q^{(i+1)}_{ux} & Q^{(i+1)}_{uu}-R
\end{array}
\right] \left[\begin{array}{c}
X_\mathcal{T}(s) \\
u_\mathcal{T}(s)
\end{array}
\right]ds.\label{eq16-b}
\end{align}
From \eqref{algorithm 1-correction}, substituting $u_\mathcal{T}=K_\mathcal{T}X_\mathcal{T}$ into \eqref{eq16-b}, we obtain
\begin{align}
&X_\mathcal{T}(t)^{\top}P^{(i+1)}X_\mathcal{T}(t)-\mathbb{E}^{\mathcal{F}_t}\left[{X_\mathcal{T}(t+\Delta t)^{\top}}P^{(i+1)}X_\mathcal{T}(t+\Delta t)\right] \nonumber\\
= & -\mathbb{E}^{\mathcal{F}_t}\int_t^{t+\Delta t}{X_\mathcal{T}(s)}^{\top}\left\{P^{(i+1)}(A+BK_\mathcal{T})+(A+BK_\mathcal{T})^{\top}P^{(i+1)}+(C+DK_\mathcal{T})^{\top}P^{(i+1)}(C+DK_\mathcal{T})\right. \nonumber\\
&\left.+\int_{\mathcal{E}}(E(e)+F(e)K_\mathcal{T})^{\top}P^{(i+1)}(E(e)+F(e)K_\mathcal{T})\lambda(de)\right\}X_\mathcal{T}(s)ds \nonumber\\
=&~\mathbb{E}^{\mathcal{F}_t}\int_t^{t+\Delta t}{X_\mathcal{T}(s)}^{\top}\left[N^{(i)}+{K_\mathcal{T}}^{\top}RK_\mathcal{T}\right]X_\mathcal{T}(s)ds. \label{eq17}
\end{align} 
Adding \eqref{eq16-b} and \eqref{eq17}, we have 
\begin{align}
&\mathbb{E}^{\mathcal{F}_t}\int_t^{t+\Delta t}\left[\begin{array}{c}
X_\mathcal{T}(s) \\
u_\mathcal{T}(s)
\end{array}
\right]^{\top}\left[\begin{array}{cc}
Q^{(i+1)}_{xx}-P^{(i+1)} & Q^{(i+1)}_{xu} \\
Q^{(i+1)}_{ux} & Q^{(i+1)}_{uu}
\end{array}
\right] \left[\begin{array}{c}
X_\mathcal{T}(s) \\
u_\mathcal{T}(s)
\end{array}
\right]ds =\mathbf{0}. \label{eq18}
\end{align}
Dividing both sides of \eqref{eq18} by $\Delta t$ and letting $\Delta t \to 0$, we obtain 
\begin{align}
&\left[\begin{array}{c}
X_\mathcal{T}(t) \\
u_\mathcal{T}(t)
\end{array}
\right]^{\top}
\left[\begin{array}{cc}
Q^{(i+1)}_{xx}-P^{(i+1)} & Q^{(i+1)}_{xu} \\
Q^{(i+1)}_{ux} & Q^{(i+1)}_{uu}
\end{array}
\right]
\left[\begin{array}{c}
X_\mathcal{T}(t) \\
u_\mathcal{T}(t)
\end{array}
\right]ds =\mathbf{0}. \label{eq18-b}
\end{align}
Similarly, we also have
\begin{align}
&\mathbb{E}^{\mathcal{F}_t}\left\{\left[\begin{array}{c}
X_\mathcal{T}(t+\Delta t) \\
u_\mathcal{T}(t+\Delta t)
\end{array}
\right]^{\top}\left[\begin{array}{cc}
Q^{(i+1)}_{xx}-P^{(i+1)} & Q^{(i+1)}_{xu} \\
Q^{(i+1)}_{ux} & Q^{(i+1)}_{uu}
\end{array}
\right] \left[\begin{array}{c}
X_\mathcal{T}(t+\Delta t) \\
u_\mathcal{T}(t+\Delta t)
\end{array}
\right]\right\}ds =\mathbf{0}. \label{eq18-c}
\end{align}

Adding \eqref{eq18-b} and \eqref{eq18-c} to \eqref{eq17}, we obtain \eqref{algorithm3-correction}. Therefore, \eqref{algorithm3-correction} is equivalent to \eqref{algorithm 1-correction}. This concludes the proof.
\end{proof}

The following proposition shows that $N^{(i+1)}$ can be updated from the Q-function matrix without using any system information.

\begin{proposition} \label{proposition5.4}
{\rm Solving for $N^{(i+1)}$ from \eqref{algorithm3-reconstruction} is equivalent to solving the following equation: 
\begin{align}
N^{(i+1)}=&-P^{(i+1)}A-A^{\top}P^{(i+1)}-C^{\top}P^{(i+1)}C-\int_{\mathcal{E}}E(e)^{\top}P^{(i+1)}E(e)\lambda(de) \nonumber\\
&+{K^{(i+1)}}^{\top}(R+D^{\top}P^{(i+1)}D+\int_{\mathcal{E}}F(e)^{\top}P^{(i+1)}F(e)\lambda(de))K^{(i+1)}. \label{algorithm 1-reconstruction}
\end{align}
}
\end{proposition}

\begin{proof}
First, rewrite \eqref{algorithm 1-reconstruction} as 
\begin{align}
N^{(i+1)} =& -P^{(i+1)}(A+BK^{(i+1)})-(A+BK^{(i+1)})^{\top}P^{(i+1)}-(C+DK^{(i+1)})^{\top}P^{(i+1)}(C+DK^{(i+1)}) \nonumber\\
&-\int_{\mathcal{E}}(E(e)+F(e)K^{(i+1)})^{\top}P^{(i+1)}(E(e)+F(e)K^{(i+1)})\lambda(de)-{K^{(i+1)}}^{\top}RK^{(i+1)}, \label{eq21}
\end{align}
where $K^{(i+1)}$ is shown in \eqref{algorithm3-update}. 
By \eqref{algorithm 1-correction}, we have
\begin{align}
&~P^{(i+1)}(A+BK^{(i+1)})+(A+BK^{(i+1)})^{\top}P^{(i+1)}+(C+DK^{(i+1)})^{\top}P^{(i+1)}(C+DK^{(i+1)}) \nonumber\\
&+\int_{\mathcal{E}}(E(e)+F(e)K^{(i+1)})^{\top}P^{(i+1)}(E(e)+F(e)K^{(i+1)})\lambda(de) \nonumber\\
=&~P^{(i+1)}(A+BK_\mathcal{T})+(A+BK_\mathcal{T})^{\top}P^{(i+1)}+(C+DK_\mathcal{T})^{\top}P^{(i+1)}(C+DK_\mathcal{T}) \nonumber\\
&+\int_{\mathcal{E}}(E(e)+F(e)K_\mathcal{T})^{\top}P^{(i+1)}(E(e)+F(e)K_\mathcal{T})\lambda(de) \nonumber\\
&+\left[\left(K^{(i+1)}-K_\mathcal{T}\right)^{\top}\left(B^{\top}P^{(i+1)}+D^{\top}P^{(i+1)}C+\int_{\mathcal{E}}{F(e)}^{\top}P^{(i+1)}E(e)\lambda(de)\right) \right. \nonumber\\
&+\left(P^{(i+1)}B+C^{\top}P^{(i+1)}D+\int_{\mathcal{E}}{E(e)}^{\top}P^{(i+1)}F(e)\lambda(de)\right)\left(K^{(i+1)}-K_\mathcal{T}\right) \nonumber\\
&+(K^{(i+1)})^{\top}D^{\top}P^{(i+1)}DK^{(i+1)}-K_\mathcal{T}^{\top}D^{\top}P^{(i+1)}DK_\mathcal{T} \nonumber\\
&\left. +\int_{\mathcal{E}}(K^{(i+1)})^{\top}{F(e)}^{\top}P^{(i+1)}F(e)K^{(i+1)}\lambda(de)-\int_{\mathcal{E}}K_\mathcal{T}^{\top}{F(e)}^{\top}P^{(i+1)}F(e)K_\mathcal{T}\lambda(de)\right] \nonumber\\
=&-N^{(i)}-{K^{(i+1)}}^{\top}RK^{(i+1)}+{K^{(i+1)}}^{\top}Q_{uu}^{(i+1)}K^{(i+1)}-K_\mathcal{T}^{\top}Q_{uu}^{(i+1)}K_\mathcal{T}+\left(K^{(i+1)}-K_\mathcal{T}\right)^{\top}Q_{ux}^{(i+1)} \nonumber\\
&+Q_{xu}^{(i+1)}\left(K^{(i+1)}-K_\mathcal{T}\right). \label{eq26} 
\end{align}
From \eqref{algorithm3-update}, we obtain
\begin{equation*}
\begin{aligned}
Q_{ux}^{(i+1)}=-Q_{uu}^{(i+1)}K^{(i+1)}.
\end{aligned}
\end{equation*}
Substituting the above equation into \eqref{eq26}, one gets
\begin{align}
&P^{(i+1)}(A+BK^{(i+1)})+(A+BK^{(i+1)})^{\top}P^{(i+1)}+(C+DK^{(i+1)})^{\top}P^{(i+1)}(C+DK^{(i+1)}) \nonumber\\
&+\int_{\mathcal{E}}(E(e)+F(e)K^{(i+1)})^{\top}P^{(i+1)}(E(e)+F(e)K^{(i+1)})\lambda(de) \nonumber\\
=&-N^{(i)}-(K_\mathcal{T}-K^{(i+1)})^{\top}Q_{uu}^{(i+1)}(K_\mathcal{T}-K^{(i+1)})-(K^{(i+1)})^{\top}RK^{(i+1)}.\label{eq27}
\end{align}
Substituting \eqref{eq27} into \eqref{eq21} yields \eqref{algorithm3-reconstruction}. 

Conversely, if $P^{(i+1)}$ is obtained from \eqref{algorithm3-reconstruction}, then \eqref{algorithm 1-correction} and \eqref{algorithm3-reconstruction} imply \eqref{algorithm 1-reconstruction}.
\end{proof}

According to Propositions \hyperref[proposition5.3]{3.1} and \hyperref[proposition5.4]{3.2}, the data-driven iterations \cref{algorithm3-correction,algorithm3-reconstruction} are equivalent to dynamic-based iterations \cref{algorithm 1-correction,algorithm 1-reconstruction}. Then we analyze the theoretical property of the proposed Algorithm \hyperref[Algorithm 3]{1}.

The following theorem shows that $K^{(i+1)}$ generated by Algorithm \hyperref[Algorithm 3]{1} is a stabilizer for all $i=0,1,2,\dots$.

\begin{theorem}[\rm Stability] \label{theorem5.1}
{\rm Given an initial matrix $N^{(0)}>\mathbf{0}$ and the stabilizer $K_{\mathcal{T}}$, every policy $K^{(i)}$, $i=1,2,\dots$, generated by \eqref{algorithm3-update} is a stabilizer for system \eqref{expert-system}.}
\end{theorem}

\begin{proof}
We prove the result by mathematical induction. Since $N^{(0)}>\mathbf{0}$ and $K_\mathcal{T}$ is a stabilizer of system \eqref{expert-system}, Lemma \hyperref[lemma 2.3]{2.1} implies that \eqref{algorithm 1-correction} admits a unique solution $P^{(1)}\in \mathbb{S}^n_{++}$. From \eqref{algorithm3-reconstruction}, we have $N^{(1)}\geq N^{(0)}>\mathbf{0}$. Then, \eqref{eq21} implies that $K^{(1)}$ is a stabilizer.

Next, suppose that $N^{(i)}>\mathbf{0}$ for some $i \geq 1$. Then, \eqref{algorithm 1-correction} admits a unique solution $P^{(i+1)}\in \mathbb{S}^n_{++}$. Since $R \in \mathbb{S}^m_{++}$ and $P^{(i+1)}\in \mathbb{S}^n_{++}$, from \eqref{algorithm3-reconstruction}, we have $N^{(i+1)} \geq N^{(i)} >\mathbf{0}$. Building upon this result, \eqref{eq21} gives
\begin{align}
&P^{(i+1)}(A+BK^{(i+1)})+(A+BK^{(i+1)})^{\top}P^{(i+1)}+(C+DK^{(i+1)})^{\top}P^{(i+1)}(C+DK^{(i+1)}) \nonumber\\
&+\int_{\mathcal{E}}(E(e)+F(e)K^{(i+1)})^{\top}P^{(i+1)}(E(e)+F(e)K^{(i+1)})\lambda(de) \nonumber\\
=&-N^{(i+1)}-(K^{(i+1)})^{\top}RK^{(i+1)} < \mathbf{0}.\label{eq27-a}
\end{align}
It follows from Lemma \hyperref[lemma 2.3]{2.1} that $K^{(i+1)}$ stabilizes system \eqref{expert-system} at each iteration of Algorithm \hyperref[Algorithm 3]{1}. 
\end{proof}

Next, we establish the convergence of Algorithm \hyperref[Algorithm 3]{1}.

\begin{theorem}[\rm Convergence and Optimality] \label{theorem4.2}
{\rm Given $R \in \mathbb{S}^m_{++}$ and $N \in \mathbb{S}^n_{++}$ such that \eqref{eq8} and \eqref{eq10} admit a unique solution $P \in \mathbb{S}^n_{++}$. If the initial $N^{(0)}$ satisfies ${\bf 0} < N^{(0)} < N$ in Algorithm \hyperref[Algorithm 3]{1}, then the sequences $\left\{P^{(i)}\right\}_{i=1}^{\infty}$, $\left\{K^{(i)}\right\}_{i=1}^{\infty}$ and $\left\{N^{(i)}\right\}_{i=1}^{\infty}$ generated by it have the following properties. }
\begin{enumerate}[1)]
\item {\rm The sequences $\left\{P^{(i)}\right\}_{i=1}^{\infty}$ and $\left\{N^{(i)}\right\}_{i=1}^{\infty}$ converge to some $P^*$ and $N^*$, respectively. Moreover, $P^*$ and $N^*$ satisfy}
\begin{align}
P^{*}&A+A^{\top}P^{*}+C^{\top}P^{*}C+\int_{\mathcal{E}}E(e)^{\top}P^{*}E(e)\lambda(de)-\left(P^{*}B+C^{\top}P^{*}D+\int_{\mathcal{E}}{E(e)}^{\top}P^{*}F(e)\lambda(de)\right) \nonumber\\
&\left(R+D^{\top}P^{*}D+\int_{\mathcal{E}}F(e)^{\top}P^{*}F(e)\lambda(de)\right)^{-1}\left(B^{\top}P^{*}+D^{\top}P^{*}C+\int_{\mathcal{E}}{F(e)}^{\top}P^{*}E(e)\lambda(de)\right) \nonumber\\
&+N^{*}=\mathbf{0}. \label{eq28}
\end{align}

\item {\rm The sequence $\left\{K^{(i)}\right\}_{i=1}^{\infty}$ converges to }
\begin{align}
K^{*}=-\left(R+D^{\top}P^{*}D+\int_{\mathcal{E}}F(e)^{\top}P^{*}F(e)\lambda(de)\right)^{-1}\left(B^{\top}P^{*}+D^{\top}P^{*}C+\int_{\mathcal{E}}F(e)^{\top}P^{*}E(e)\lambda(de)\right). \label{eq29}
\end{align}
\end{enumerate}
{\rm Furthermore, we have $K^{*}=K_\mathcal{T}$, and $(P,N)=(P^{*}, N^{*})$ satisfy \eqref{eq8} and \eqref{eq10}}.
\end{theorem}

\begin{proof}
From \eqref{algorithm 1-correction}, we obtain
\begin{align}
N^{(i-1)}-N^{(i)} =&~\left(P^{(i+1)}-P^{(i)}\right)(A+BK_\mathcal{T})+(A+BK_\mathcal{T})^{\top}\left(P^{(i+1)}-P^{(i)}\right) \nonumber\\
&+(C+DK_\mathcal{T})^{\top}\left(P^{(i+1)}-P^{(i)}\right)(C+DK_\mathcal{T}) \nonumber\\
&+\int_{\mathcal{E}}(E(e)+F(e)K_\mathcal{T})^{\top}\left(P^{(i+1)}-P^{(i)}\right)(E(e)+F(e)K_\mathcal{T})\lambda(de).
\end{align}
Since $K_\mathcal{T}$ is a stabilizer and $N^{(i-1)}-N^{(i)}\leq \mathbf{0}$, from Lemma \hyperref[lemma 2.3]{2.1}, we have $\mathbf{0}<P^{(i)}\leq P^{(i+1)}$ for $i\geq 1$. 
Rewrite Lyapunov equation \eqref{eq10} to obtain
\begin{align}
&P(A+BK^{(i+1)})+(A+BK^{(i+1)})^{\top}P+(C+DK^{(i+1)})^{\top}P(C+DK^{(i+1)}) \nonumber\\
&+\int_{\mathcal{E}}(E(e)+F(e)K^{(i+1)})^{\top}P(E(e)+F(e)K^{(i+1)})\lambda(de)+N \nonumber\\
&-(K_\mathcal{T}-K^{(i+1)})^{\top}(R+D^{\top}PD+\int_{\mathcal{E}}F(e)^{\top}PF(e)\lambda(de))(K_\mathcal{T}-K^{(i+1)})+{K^{(i+1)}}^{\top}RK^{(i+1)}=\mathbf{0}. \label{eq30}
\end{align}
Combining \eqref{algorithm3-update}, we can reformulate \eqref{algorithm 1-reconstruction} into the following equation
\begin{align}
 &P^{(i+1)}(A+BK^{(i+1)})+(A+BK^{(i+1)})^{\top}P^{(i+1)}+(C+DK^{(i+1)})^{\top}P^{(i+1)}(C+DK^{(i+1)}) \nonumber\\
&+\int_{\mathcal{E}}(E(e)+F(e)K^{(i+1)})^{\top}P^{(i+1)}(E(e)+F(e)K^{(i+1)})\lambda(de)+N^{(i+1)}+{K^{(i+1)}}^{\top}RK^{(i+1)} = \mathbf{0}. \label{eq31}
\end{align}
We next prove by mathematical induction that $\left\{P^{(i)}\right\}_{i=1}^{\infty}$ and $\left\{N^{(i)}\right\}_{i=1}^{\infty}$ are bounded above by $P$ and $N$. 

For $i=0$, since $N^{(0)}<N$ and $K_\mathcal{T}$ is a stabilizer, subtracting \eqref{algorithm 1-correction} from \eqref{eq10} gives $P^{(1)}<P$. Moreover, by Theorem \hyperref[theorem5.1]{3.1}, $K^{(1)}$ is also a stabilizer of system \eqref{expert-system}. Subtracting \eqref{eq31} from \eqref{eq30} implies $N^{(1)}<N$. Suppose that $N^{(i)} < N$ for some $i\geq 1$. Subtracting \eqref{algorithm 1-correction} from \eqref{eq10} yields
\begin{align*}
\left(P-P^{(i+1)}\right)&(A+BK_\mathcal{T})+(A+BK_\mathcal{T})^{\top}\left(P-P^{(i+1)}\right)+(C+DK_\mathcal{T})^{\top}\left(P-P^{(i+1)}\right)(C+DK_\mathcal{T}) \nonumber\\
&+\int_{\mathcal{E}}(E(e)+F(e)K_\mathcal{T})^{\top}\left(P-P^{(i+1)}\right)(E(e)+F(e)K_\mathcal{T})\lambda(de)+N-N^{(i)}=\mathbf{0},
\end{align*}
which shows $P-P^{(i+1)}\in \mathbb{S}^n_{++}$. Subtracting \eqref{eq31} from \eqref{eq30}, by Theorem \hyperref[theorem5.1]{3.1}, $K^{(i+1)}$ is a stabilizer, then
\begin{align}
\mathbf{0}>&(P-P^{(i+1)})(A+BK^{(i+1)})+(A+BK^{(i+1)})^{\top}(P-P^{(i+1)}) \nonumber\\
&+(C+DK^{(i+1)})^{\top}(P-P^{(i+1)})(C+DK^{(i+1)})\nonumber\\
&+\int_{\mathcal{E}}(E(e)+F(e)K^{(i+1)})^{\top}(P-P^{(i+1)})(E(e)+F(e)K^{(i+1)})\lambda(de) \nonumber\\
&=~N^{(i+1)}-N+(K_\mathcal{T}-K^{(i+1)})^{\top}\left(R+D^{\top}PD+\int_{\mathcal{E}}F(e)^{\top}PF(e)\lambda(de)\right)(K_\mathcal{T}-K^{(i+1)}). \label{eq32}
\end{align}
Since $R \in \mathbb{S}^m_{++}$, $P \in \mathbb{S}^n_{++}$, and $P-P^{(i+1)}\in \mathbb{S}^n_{++}$, we have $N^{(i+1)} < N$. Thus the generated sequences $\left\{P^{(i)}\right\}_{i=1}^{\infty}$ and $\left\{N^{(i)}\right\}_{i=1}^{\infty}$ are monotonically nondecreasing and bounded above by $P$ and $N$, respectively.

Next, we show that the limits $P^*$, $K^*$ and $N^{*}$ satisfy \eqref{eq28} and \eqref{eq29}. Substituting \eqref{algorithm 1-reconstruction} into \eqref{algorithm 1-correction} yields
\begin{align}
&~P^{(i+1)}A+A^{\top}P^{(i+1)}+C^{\top}P^{(i+1)}C+\int_{\mathcal{E}}E(e)^{\top}P^{(i+1)}E(e)\lambda(de) \nonumber\\
&+K_\mathcal{T}^{\top}\left(B^{\top}P^{(i+1)}+D^{\top}P^{(i+1)}C+\int_{\mathcal{E}}{F(e)}^{\top}P^{(i+1)}E(e)\lambda(de)\right) \nonumber\\
&+\left(P^{(i+1)}B+C^{\top}P^{(i+1)}D+\int_{\mathcal{E}}{E(e)}^{\top}P^{(i+1)}F(e)\lambda(de)\right)K_\mathcal{T} \nonumber\\
&+K_\mathcal{T}^{\top}\left(R+D^{\top}P^{(i+1)}D+\int_{\mathcal{E}}{F(e)}^{\top}P^{(i+1)}F(e)\lambda(de)\right)K_\mathcal{T} \nonumber\\
&=~P^{(i)}A+A^{\top}P^{(i)}+C^{\top}P^{(i)}C+\int_{\mathcal{E}}E(e)^{\top}P^{(i)}E(e)\lambda(de) \nonumber\\
&-{K^{(i)}}^{\top}\left(R+D^{\top}P^{(i)}D+\int_{\mathcal{E}}{F(e)}^{\top}P^{(i)}F(e)\lambda(de)\right)K^{(i)}. \label{star3}
\end{align}

Taking the limit on both sides of \eqref{star3}, we have
\begin{align*}
(K_\mathcal{T}-K^{*})^{\top}(R+D^{\top}P^{*}D+\int_{\mathcal{E}}F(e)^{\top}P^{*}F(e)\lambda(de))(K_\mathcal{T}-K^{*})=\mathbf{0},
\end{align*}
where $K^{*}$ satisfies \eqref{eq29}. Since $P^{*}\in \mathbb{S}^n_{++}$, it follows that $K^*=K_\mathcal{T}$. From \eqref{algorithm 1-reconstruction}, the limit $P^{*}$ and $N^{*}$ satisfy \eqref{eq28}.

Using $K^*=K_\mathcal{T}$, \eqref{eq28} can be rewritten as follows
\[
\begin{aligned}
&P^*(A+BK_\mathcal{T})+(A+BK_\mathcal{T})^{\top}P^*+(C+DK_\mathcal{T})^{\top}P^{*}(C+DK_\mathcal{T})\\
&+\int_{\mathcal{E}}(E(e)+F(e)K_\mathcal{T})^{\top}P^{*}(E(e)+F(e)K_\mathcal{T})\lambda(de)+N^{*}+{K_\mathcal{T}}^{\top}RK_\mathcal{T}=\mathbf{0},
\end{aligned}
\]
which coincides with \eqref{eq10}. Meanwhile, \eqref{eq28} is actually \eqref{eq8}. Finally, the limits $P^{*}$ and $N^{*}$ satisfy \eqref{eq8} and \eqref{eq10}. This completes the proof.
\end{proof}

\subsection{Implementation of Algorithm 1}

~~~~In this subsection, we implement the inverse Q-learning algorithm using vectorization and the Kronecker product, based on the expert's demonstrated state trajectories and the target control gain $K_\mathcal{T}$. Denote the actual input by $\widetilde{u}_\mathcal{T}=K_\mathcal{T}X_\mathcal{T}+e_\mathcal{T}$, where $e_\mathcal{T}$ is a bounded signal satisfying the persistence of excitation (PE) condition; see \cite{lian2024inverse, rizvi2019reinforcement, Xue-2021inverse}.

To calculate $\mathbf{Q}^{(i+1)}$ from \eqref{algorithm3-correction}, we define
\[
\begin{aligned}
\theta_{\mathcal{T}}(t):=&~\left(\left[X_\mathcal{T}(t),\widetilde{u}_\mathcal{T}(t)\right]\otimes\left[X_\mathcal{T}(t),\widetilde{u}_\mathcal{T}(t)\right]\right)^{\top},\\
\Theta_{\mathcal{T}}:=&~\mathbb{E}^{\mathcal{F}_t}\left[\theta_{\mathcal{T}}(t)-\theta_{\mathcal{T}}(t+\Delta t), \dots , \theta_{\mathcal{T}}(t+(l-1)\Delta t)-\theta_{\mathcal{T}}(t+l\Delta t)\right]^{\top},\\
\Lambda_{\mathcal{T}}^{(i)}:=&~\mathbb{E}^{\mathcal{F}_t}\left[\int_t^{t+\Delta t}{X_\mathcal{T}(s)}^{\top}(N^{(i)}+{K_\mathcal{T}}^{\top}RK_\mathcal{T})X_\mathcal{T}(s)ds, \dots, \right.\\
&~\qquad\qquad\left.\int_{t+(l-1)\Delta t}^{t+l\Delta t}{X_\mathcal{T}(s)}^{\top}(N^{(i)}+{K_\mathcal{T}}^{\top}RK_\mathcal{T})X_\mathcal{T}(s)ds\right]^{\top}.
\end{aligned}
\]
Note that $l$ should be no less than the number of unknown parameters of $\mathbf{Q}^{(i+1)}$, i.e. $l\geq \frac{n(n+1)}{2}+mn+\frac{m(m+1)}{2}$. 

Now we impose the following rank condition on the data matrix $\Theta_{\mathcal{T}}$. 
\begin{assumption} \label{lemma5.1}
{\rm There exists a constant $k_0 \geq \frac{n(n+1)}{2}+mn+\frac{m(m+1)}{2}$ such that for all $l\geq k_0$,}
\begin{equation}\label{rank condition2}
\begin{aligned}
rank(\Theta_{\mathcal{T}})=\frac{n(n+1)}{2}+mn+\frac{m(m+1)}{2}.
\end{aligned}
\end{equation}
\end{assumption}
\noindent Assumption \ref{lemma5.1} can be viewed as a PE condition. It should be emphasized that $e_\mathcal{T}$ is not merely introduced as general exploration noise. In fact, general exploration noise is used to perturb the control input and enrich the collected trajectories, while probing noise is imposed to satisfy the PE condition such that $\Theta_{\mathcal{T}}$ has full column rank and thus the Q-function matrix can be uniquely identified by the batch least squares method.

Under Assumption \ref{lemma5.1}, the matrix $\Theta_{\mathcal{T}}^{\top}\Theta_{\mathcal{T}}$ is invertible. Hence, $\mathbf{Q}^{(i+1)}$ is uniquely determined via the batch least squares method as
\begin{equation} \label{eq23}
\begin{aligned}
vec(\mathbf{Q}^{(i+1)})
=(\Theta_{\mathcal{T}}^{\top}\Theta_{\mathcal{T}})^{-1}\Theta_{\mathcal{T}}^{\top}\Lambda_{\mathcal{T}}^{(i)}.
\end{aligned} 
\end{equation}
By \eqref{algorithm3-reconstruction}, $N^{(i+1)}$ is updated directly. 

In summary, this section develops a model-free off-policy inverse Q-learning algorithm for solving Problem (ISLQ) using expert trajectories. By introducing the Q-function, the IOC problem is addressed within an IRL framework, where data-driven equations are constructed independently of the system parameters and jump intensity. The expert's demonstrated state trajectories are used to calculate the Q-function matrix, then the cost weight is updated without requiring additional trajectory data. In the next section, we further develop a complementary model-free off-policy inverse policy iteration algorithm based on the learner's trajectories.

\section{Model-Free Off-Policy Inverse Policy Iteration Algorithm for Problem (ISLQ)} \label{section4}
~~~~Unlike the inverse Q-learning algorithm in the previous section using data collected from the target expert, this section proposes a model-free off-policy inverse policy iteration algorithm based on trajectories generated by the learner. Specifically, when the expert's state trajectories \(X_{\mathcal T}(\cdot)\) are unavailable, the learner can collect state and control trajectories by applying an arbitrary initial stabilizing behavior policy \(K_\mathcal{L}\). This behavior policy is used only for data collection and is decoupled from the feedback policies updated during iteration. Based on the learner's trajectories, the cost weights are iteratively updated without requiring knowledge of the system dynamics or jump intensity.

\subsection{Model-Free Off-Policy Inverse Policy Iteration Algorithm Design}
~~~~We now introduce the model-free off-policy inverse policy iteration algorithm.

\begin{algorithm}[H]
\caption{\rm Model-Free Off-Policy Inverse Policy Iteration Algorithm} \label{Algorithm 2} 
1: {\bf Initialization}: For a given $R>{\bf 0}$, choose an initial $N^{(0)}>{\bf 0}$. Select an arbitrary stabilizer $K_\mathcal{L}$ and collect the learner's data $X_\mathcal{L}(s)$ by running system \eqref{expert-system} with the behavior policy $u_\mathcal{L}(\cdot)=K_\mathcal{L}X_\mathcal{L}(\cdot)+e_{\mathcal{L}}(\cdot)$, where $e_{\mathcal{L}}(\cdot)$ is a probing noise. Let $i=0$ and $\varepsilon>0$.\\
2: {\bf do} $\{$\\
3: \textbf{Policy Evaluation}: Solve for $P^{(i+1)}$, $\mathbb{B}^{(i+1)}$, and $\mathbb{D}^{(i+1)}$ via the identity
\begin{align} \label{algorithm2-correction}
&~X_\mathcal{L}(t)^{\top}P^{(i+1)}X_\mathcal{L}(t)-\mathbb{E}^{\mathcal{F}_t}\left[{X_\mathcal{L}(t+\Delta t)^{\top}}P^{(i+1)}X_\mathcal{L}(t+\Delta t)\right] \nonumber\\
&+2\mathbb{E}^{\mathcal{F}_t}\int_t^{t+\Delta t}(u_\mathcal{L}(s)-K_\mathcal{T}X_\mathcal{L}(s))^{\top}\mathbb{B}^{(i+1)}X_\mathcal{L}(s)ds \nonumber\\
&+\mathbb{E}^{\mathcal{F}_t}\int_t^{t+\Delta t}{u_\mathcal{L}(s)}^{\top}\mathbb{D}^{(i+1)}u_\mathcal{L}(s)ds \nonumber\\
&-\mathbb{E}^{\mathcal{F}_t}\int_t^{t+\Delta t}{X_\mathcal{L}(s)}^{\top}{K_\mathcal{T}}^{\top}\mathbb{D}^{(i+1)}K_\mathcal{T}X_\mathcal{L}(s)ds \nonumber\\
=&~\mathbb{E}^{\mathcal{F}_t}\int_t^{t+\Delta t}{X_\mathcal{L}(s)}^{\top}(N^{(i)}+{K_\mathcal{T}}^{\top}RK_\mathcal{T})X_\mathcal{L}(s)ds.
\end{align}
4: \textbf{Policy Improvement}: Update $K^{(i+1)}$ by\\
\begin{equation} \label{algorithm2-update}
\begin{aligned}
K^{(i+1)}=&-(R+\mathbb{D}^{(i+1)})^{-1}\mathbb{B}^{(i+1)}. 
\end{aligned}
\end{equation}
5: \textbf{Weight Update}: Update $N^{(i+1)}$ via the equation\\
\begin{align} \label{algorithm2-reconstruction}
\mathbb{E}^{\mathcal{F}_t}\int_t^{t+\Delta t}{X_\mathcal{L}(s)^{\top}}N^{(i+1)}X_\mathcal{L}(s)ds
=&~X_\mathcal{L}(t)^{\top}P^{(i+1)}X_\mathcal{L}(t)-\mathbb{E}^{\mathcal{F}_t}\left[{X_\mathcal{L}(t+\Delta t)^{\top}}P^{(i+1)}X_\mathcal{L}(t+\Delta t)\right] \nonumber\\
&+2\mathbb{E}^{\mathcal{F}_t}\int_t^{t+\Delta t}u_\mathcal{L}(s)^{\top}\mathbb{B}^{(i+1)}X_\mathcal{L}(s)ds \nonumber\\
&+\mathbb{E}^{\mathcal{F}_t}\int_t^{t+\Delta t}{u_\mathcal{L}(s)}^{\top}\mathbb{D}^{(i+1)}u_\mathcal{L}(s)ds \nonumber\\
&+\mathbb{E}^{\mathcal{F}_t}\int_t^{t+\Delta t}{X_\mathcal{L}(s)}^{\top}{K^{(i+1)}}^{\top}(R+\mathbb{D}^{(i+1)})K^{(i+1)}X_\mathcal{L}(s)ds.
\end{align}
6: $i\leftarrow i+1$.\\
7: $\}$ {\bf until} $\|N^{(i+1)}-N^{(i)}\|\leq \varepsilon$. 
\end{algorithm}

Next, we demonstrate the stability and convergence of Algorithm \hyperref[Algorithm 2]{2}.
\begin{theorem} \label{proposition5.1}
{\rm\begin{enumerate} [1)]
\item All the feedback gains $\left\{K^{(i)}\right\}_{i=1}^{\infty}$ updated by \eqref{algorithm2-update} are stabilizers.
\item The sequences $\left\{P^{(i)}\right\}_{i=1}^{\infty}$, $\left\{K^{(i)}\right\}_{i=1}^{\infty}$, and $\left\{N^{(i)}\right\}_{i=1}^{\infty}$ generated by Algorithm \hyperref[Algorithm 2]{2} converge to $P^*$, $K^*$, and $N^*$ defined in \eqref{eq28} and \eqref{eq29}.
\end{enumerate}}
\end{theorem}

\begin{proof}
First, we prove that solving for $P^{(i+1)}$ by \eqref{algorithm2-correction} is equivalent to solving \eqref{algorithm 1-correction}. Applying $\mathrm{It\hat{o}}$'s formula to $X_\mathcal{L}(s)^{\top}P^{(i+1)}X_\mathcal{L}(s)$. 
\begin{align} 
&~d[X_\mathcal{L}(s)^{\top}P^{(i+1)}X_\mathcal{L}(s)] \nonumber\\
=&~\left\{X_\mathcal{L}(s)^{\top}\left(P^{(i+1)}A+A^{\top}P^{(i+1)}+C^{\top}P^{(i+1)}C+\int_{\mathcal{E}}{E(e)}^{\top}P^{(i+1)}{E(e)}\lambda(de)\right)X_\mathcal{L}(s)\right.\nonumber \\
&+2u_\mathcal{L}(s)^{\top}\left(B^{\top}P^{(i+1)}+D^{\top}P^{(i+1)}C+\int_{\mathcal{E}}{F(e)}^{\top}P^{(i+1)}{E(e)}\lambda(de)\right)X_\mathcal{L}(s) \nonumber \\
&\left.+u_\mathcal{L}(s)^{\top}\left(D^{\top}P^{(i+1)}D+\int_{\mathcal{E}}{F(e)}^{\top}P^{(i+1)}{F(e)}\lambda(de)\right)u_\mathcal{L}(s)\right\}ds \nonumber \\
&+\left\{\dots\right\}dW(s)+\left\{\dots\right\}\widetilde{N}(ds,de). \label{eq18-a}
\end{align}
Integrating \eqref{eq18-a} from $t$ to $t+\Delta t$ and then taking the conditional expectation $\mathbb{E}^{\mathcal{F}_t}$, we obtain
\begin{align}
&~\mathbb{E}^{\mathcal{F}_t}\left[{X_\mathcal{L}(t+\Delta t)^{\top}}P^{(i+1)}X_\mathcal{L}(t+\Delta t)\right]-X_\mathcal{L}(t)^{\top}P^{(i+1)}X_\mathcal{L}(t) \nonumber\\
=&~\mathbb{E}^{\mathcal{F}_t}\int_t^{t+\Delta t}X_\mathcal{L}(s)^{\top}\left[P^{(i+1)}A+A^{\top}P^{(i+1)}+C^{\top}P^{(i+1)}C+\int_{\mathcal{E}}{E(e)}^{\top}P^{(i+1)}{E(e)}\lambda(de)\right]X_\mathcal{L}(s)ds \nonumber\\
&+2\mathbb{E}^{\mathcal{F}_t}\int_t^{t+\Delta t}u_\mathcal{L}(s)^{\top}\mathbb{B}^{(i+1)}X_\mathcal{L}(s)ds+\mathbb{E}^{\mathcal{F}_t}\int_t^{t+\Delta t}u_\mathcal{L}(s)^{\top}\mathbb{D}^{(i+1)}u_\mathcal{L}(s)ds, \label{eq20-a}
\end{align}
where 
\begin{align*}
\mathbb{B}^{(i+1)}=&B^{\top}P^{(i+1)}+D^{\top}P^{(i+1)}C +\int_{\mathcal{E}}{F(e)}^{\top}P^{(i+1)}{E(e)}\lambda(de),\\
\mathbb{D}^{(i+1)}=&D^{\top}P^{(i+1)}D+\int_{\mathcal{E}}{F(e)}^{\top}P^{(i+1)}{F(e)}\lambda(de).
\end{align*}
Substituting the equality \eqref{algorithm 1-correction} into \eqref{eq20-a}, it follows that
\begin{align*}
&\mathbb{E}^{\mathcal{F}_t}\left[{X_\mathcal{L}(t+\Delta t)^{\top}}P^{(i+1)}X_\mathcal{L}(t+\Delta t)\right]-X_\mathcal{L}(t)^{\top}P^{(i+1)}X_\mathcal{L}(t) \\
=&-\mathbb{E}^{\mathcal{F}_t}\int_t^{t+\Delta t}X_\mathcal{L}(s)^{\top}(N^{(i)}+{K_\mathcal{T}}^{\top}RK_\mathcal{T})X_\mathcal{L}(s)ds \\
&+2\mathbb{E}^{\mathcal{F}_t}\int_t^{t+\Delta t}\left(u_\mathcal{L}(s)-K_\mathcal{T}X_\mathcal{L}(s)\right)^{\top}\mathbb{B}^{(i+1)}X_\mathcal{L}(s)ds \\
&+\mathbb{E}^{\mathcal{F}_t}\int_t^{t+\Delta t}u_\mathcal{L}(s)^{\top}\mathbb{D}^{(i+1)}u_\mathcal{L}(s)ds-\mathbb{E}^{\mathcal{F}_t}\int_t^{t+\Delta t}X_\mathcal{L}(s)^{\top}K^{\top}_\mathcal{T}\mathbb{D}^{(i+1)}K_\mathcal{T}X_\mathcal{L}(s)ds,
\end{align*}
which confirms \eqref{algorithm2-correction}.

Conversely, if $P^{(i+1)} \in \mathbb{S}^{n}$ is the solution of \eqref{algorithm2-correction}, for any $\tau >t$, a calculation similar to \eqref{eq20-a} yields
\begin{equation*} 
\begin{aligned}
\mathbb{E}^{\mathcal{F}_{\tau}}&\int_{\tau}^{\tau +\Delta t}{X_\mathcal{L}(s)^{\top}}\left\{P^{(i+1)}A+A^{\top}P^{(i+1)}+C^{\top}P^{(i+1)}C+\int_{\mathcal{E}}E(e)^{\top}P^{(i+1)}E(e)\lambda(de)\right\}X_\mathcal{L}(s)ds\\
&+2\mathbb{E}^{\mathcal{F}_{\tau}}\int_{\tau}^{\tau +\Delta t}{X_\mathcal{L}(s)^{\top}}{K_\mathcal{T}}^{\top}\mathbb{B}^{(i+1)}X_\mathcal{L}(s)ds +\mathbb{E}^{\mathcal{F}_{\tau}}\int_{\tau}^{\tau +\Delta t}{X_\mathcal{L}(s)}^{\top}K^{\top}_\mathcal{T}\mathbb{D}^{(i+1)}K_\mathcal{T}X_\mathcal{L}(s)ds\\
& +\mathbb{E}^{\mathcal{F}_{\tau}}\int_{\tau}^{\tau +\Delta t}{X_\mathcal{L}(s)}^{\top}\left\{N^{(i)}+{K_\mathcal{T}}^{\top}RK_\mathcal{T}\right\}X_\mathcal{L}(s)ds=0.
\end{aligned}
\end{equation*}
Dividing both sides of the above equation by $\Delta t$ and taking the limit as $\Delta t \to 0$, we deduce that 
\begin{equation*}
\begin{aligned}
x_{\tau}^{\top}&\biggl\{ P^{(i+1)}(A+BK_\mathcal{T})+(A+BK_\mathcal{T})^{\top}P^{(i+1)}+(C+DK_\mathcal{T})^{\top}P^{(i+1)}(C+DK_\mathcal{T})\\
&+\int_{\mathcal{E}}(E(e)+F(e)K_\mathcal{T})^{\top}P^{(i+1)}(E(e)+F(e)K_\mathcal{T})\lambda(de)+N^{(i)}+{K_\mathcal{T}}^{\top}RK_\mathcal{T} \biggr\}x_{\tau}=0,
\end{aligned}
\end{equation*}
where $x_{\tau}$ denotes the learner's state at time $\tau$. Since $x_{\tau}$ can take any value in $\mathbb{R}^{n}$, we have \eqref{algorithm 1-correction}.

Next, we show that solving \eqref{algorithm2-reconstruction} for $N^{(i+1)}$ yields the same result as solving \eqref{algorithm 1-reconstruction}. Multiplying both sides of \eqref{algorithm 1-reconstruction} by nonzero $X_\mathcal{L}(s)$, \eqref{algorithm 1-reconstruction} can be rewritten as
\begin{align}
&{X_\mathcal{L}(s)}^{\top}N^{(i+1)}X_\mathcal{L}(s) \nonumber\\
=&-{X_\mathcal{L}(s)}^{\top}\left(P^{(i+1)}A+A^{\top}P^{(i+1)}+C^{\top}P^{(i+1)}C+\int_{\mathcal{E}}{E(e)}^{\top}P^{(i+1)}{E(e)}\lambda(de)\right)X_\mathcal{L}(s) \nonumber\\
&+{X_\mathcal{L}(s)}^{\top}{K^{(i+1)}}^{\top}\left(R+D^{\top}P^{(i+1)}D+\int_{\mathcal{E}}F(e)^{\top}P^{(i+1)}F(e)\lambda(de)\right)K^{(i+1)}X_\mathcal{L}(s). \label{eq19-a}
\end{align}
By a calculation similar to that in \eqref{eq20-a}, we have
\begin{align}
&\mathbb{E}^{\mathcal{F}_t}\int_t^{t+\Delta t}{X_\mathcal{L}(s)}^{\top}N^{(i+1)}X_\mathcal{L}(s)ds \nonumber\\
=&-\mathbb{E}^{\mathcal{F}_t}\int_t^{t+\Delta t}{X_\mathcal{L}(s)}^{\top}\left[P^{(i+1)}A+A^{\top}P^{(i+1)}+C^{\top}P^{(i+1)}C+\int_{\mathcal{E}}{E(e)}^{\top}P^{(i+1)}{E(e)}\lambda(de)\right]X_\mathcal{L}(s)ds \nonumber\\
&+\mathbb{E}^{\mathcal{F}_t}\int_t^{t+\Delta t}{X_\mathcal{L}(s)}^{\top}{K^{(i+1)}}^{\top}\left(R+\mathbb{D}^{(i+1)}\right)K^{(i+1)}X_\mathcal{L}(s)ds. \label{eq21-a}
\end{align} 
Combining \eqref{eq20-a} and \eqref{eq21-a} yields \eqref{algorithm2-reconstruction}. Similarly, if $P^{(i+1)}\in \mathbb{S}^{n}$ is the solution of \eqref{algorithm2-reconstruction}, it is readily verified that \eqref{algorithm 1-reconstruction} holds. 

Thus, the iterations \cref{algorithm2-correction,algorithm2-reconstruction} are equivalent to \eqref{algorithm 1-correction} and \eqref{algorithm 1-reconstruction}, respectively. Consequently, Algorithm \hyperref[Algorithm 2]{2} is equivalent to Algorithm \hyperref[Algorithm 3]{1} and therefore has the same theoretical properties. This concludes the proof.
\end{proof}

\subsection{Implementation of Algorithm 2}
 By \cite{Li-Li,murray2002adaptive}, there exists a matrix $\mathcal{M} \in \mathbb{R}^{{{n^2} \times \frac{1}{2}}n(n+1)}$ with $rank(\mathcal{M})=\frac{n(n+1)}{2}$ such that $vec(P)=\mathcal{M}vec^{+}(P)$ for any $P\in \mathbb{S}^n$. We define $\overline{L}:=\left(L^{\top}\otimes L^{\top}\right) \mathcal{M} \in \mathbb{R}^{{{n^2} \times \frac{1}{2}}n(n+1)}$ such that $L^{\top}\otimes L^{\top}vec(P)={\overline{L}}vec^{+}(P)$ with $L\in \mathbb{R}^{n}$. In particular, for any vector $x = \left[x_1, x_2, \cdots, x_n \right]^{\top} \in \mathbb{R}^{n}$, we have
\[
\begin{aligned}
\bar{x}=&[x_{1}^2, x_{1}x_{2},\dots ,x_{1}x_{n},x^2_{2},x_{2}x_{3},\dots ,x_{n-1}x_{n},x_{n}^2]^{\top},
\end{aligned}
\]
where $\bar{x} \in \mathbb{R}^{{\frac{1}{2}}n(n+1)}$.

In \eqref{algorithm2-correction}, noting that $K_\mathcal{T}X_\mathcal{L}(s)$ and $K^{(i+1)}X_\mathcal{L}(s)$ are two column vectors, one gets
\[
\begin{aligned}
    {X_\mathcal{L}(s)}^{\top}{K_\mathcal{T}}^{\top}\mathbb{D}^{(i+1)}K_\mathcal{T}X_\mathcal{L}(s)=& {\overline{K_\mathcal{T}X_\mathcal{L}(s)}}^{\top}vec^{+}(\mathbb{D}^{(i+1)}),\\
    {X_\mathcal{L}(s)}^{\top}{K^{(i+1)}}^{\top}(R+\mathbb{D}^{(i+1)})K^{(i+1)}X_\mathcal{L}(s)=& {\overline{K^{(i+1)}X_\mathcal{L}(s)}}^{\top}vec^{+}(R+\mathbb{D}^{(i+1)}).
\end{aligned}
\]
By the Kronecker product theory, \eqref{algorithm2-correction} can be rewritten as
\begin{align*}
    &\mathrm{I}\times vec^{+}(P^{(i+1)}) + 2 \mathrm{II}\times vec(\mathbb{B}^{(i+1)}) +\mathrm{III} \times vec^{+}(\mathbb{D}^{(i+1)}) \nonumber \\
    =~& \mathrm{IV} \times vec(N^{(i)}+{K_\mathcal{T}}^{\top}RK_\mathcal{T}),
\end{align*}
where 
\begin{align*}
    \mathrm{I} &=~ \begin{bmatrix}
        \mathbb{E}^{\mathcal{F}_t}\left[\bar{X}_{\mathcal{L}}(t)^{\top}-\bar{X}_{\mathcal{L}}(t+\Delta t)^{\top}\right] \\
        \vdots\\
        \mathbb{E}^{\mathcal{F}_t}\left[\bar{X}_{\mathcal{L}}(t+(l-1)\Delta t)^{\top}-\bar{X}_{\mathcal{L}}(t+l\Delta t)^{\top}\right] \\
    \end{bmatrix}, \\
    \mathrm{II} &=~\begin{bmatrix}
        \mathbb{E}^{\mathcal{F}_t}\left[\int_t^{t+\Delta t}X_\mathcal{L}(s)^{\top}\otimes u_\mathcal{L}(s)^{\top}ds\right] \\
        \vdots\\
        \mathbb{E}^{\mathcal{F}_t}\left[\int_{t+(l-1)\Delta t}^{t+l\Delta t}X_\mathcal{L}(s)^{\top}\otimes u_\mathcal{L}(s)^{\top}ds\right] \\
    \end{bmatrix} - \begin{bmatrix}
        \mathbb{E}^{\mathcal{F}_t}\left[\int_t^{t+\Delta t}X_\mathcal{L}(s)^{\top}\otimes X_\mathcal{L}(s)^{\top}ds\right] \\
        \vdots\\
        \mathbb{E}^{\mathcal{F}_t}\left[\int_{t+(l-1)\Delta t}^{t+l\Delta t}X_\mathcal{L}(s)^{\top}\otimes X_\mathcal{L}(s)^{\top}ds\right] \\
    \end{bmatrix}\times \left(I_n\otimes K_\mathcal{T}^{\top}\right) \\
    &\coloneqq~\mathrm{II}_{xu}-\mathrm{II}_{xx}\times(I_{n}\otimes K_\mathcal{T}^{\top}), \\
    \mathrm{III} &=~ \begin{bmatrix}
        \mathbb{E}^{\mathcal{F}_t}\left[\int_t^{t+\Delta t}\bar{u}_{\mathcal{L}}(s)^{\top}ds\right] \\
        \vdots \\
        \mathbb{E}^{\mathcal{F}_t}\left[\int_{t+(l-1)\Delta t}^{t+l\Delta t}\bar{u}_{\mathcal{L}}(s)^{\top}ds\right]
    \end{bmatrix} - \begin{bmatrix}
        \mathbb{E}^{\mathcal{F}_t}\left[\int_t^{t+\Delta t}{\overline{K_\mathcal{T}X_\mathcal{L}(s)}}^{\top}ds\right] \\
        \vdots \\
        \mathbb{E}^{\mathcal{F}_t}\left[\int_{t+(l-1)\Delta t}^{t+l\Delta t}{\overline{K_\mathcal{T}X_\mathcal{L}(s)}}^{\top}ds\right]
    \end{bmatrix}\coloneqq ~\mathrm{III}_{\bar{u}}-\mathrm{III}_{\overline{K_\mathcal{T}X}}, \\
    \mathrm{IV} &=~ \begin{bmatrix}
        \mathbb{E}^{\mathcal{F}_t}\left[\int_t^{t+\Delta t}X_\mathcal{L}(s)^{\top}\otimes X_\mathcal{L}(s)^{\top}ds\right] \\
        \vdots\\
        \mathbb{E}^{\mathcal{F}_t}\left[\int_{t+(l-1)\Delta t}^{t+l\Delta t}X_\mathcal{L}(s)^{\top}\otimes X_\mathcal{L}(s)^{\top}ds\right] \\
    \end{bmatrix},
\end{align*}
where $l$ denotes the total number of data groups.
Using the above operators, we construct the system matrix $\varPhi_{p}$ and the iterative vector $\varPsi_{p}^{(i)}$ satisfying \eqref{algorithm2-correction} as follows:
\[
\begin{aligned}
\varPhi_{p}\coloneqq &~\begin{bmatrix}
\mathrm{I}, \mathrm{II}, \mathrm{III}
\end{bmatrix},\\
\varPsi_{p}^{(i)}\coloneqq &~\mathrm{IV} \times vec(N^{(i)}+{K_\mathcal{T}}^{\top}RK_\mathcal{T}).
\end{aligned}
\]
Then, \eqref{algorithm2-correction} becomes
\[
\begin{aligned}
\varPhi_{p}{
\begin{bmatrix}
vec^{+}(P^{(i+1)})\\
vec(\mathbb{B}^{(i+1)})\\
vec^{+}(\mathbb{D}^{(i+1)})
\end{bmatrix}
}=\varPsi_{p}^{(i)}.
\end{aligned}
\]

If $\varPhi_{p}$ has full column rank, $P^{(i+1)}$, $\mathbb{B}^{(i+1)}$, and $\mathbb{D}^{(i+1)}$ can be uniquely determined by the batch least squares method, and the above equation becomes
\begin{equation} \label{eq23-a}
\begin{aligned}
\begin{bmatrix}
vec^{+}(P^{(i+1)})\\
vec(\mathbb{B}^{(i+1)})\\
vec^{+}(\mathbb{D}^{(i+1)})
\end{bmatrix}
=(\varPhi_{p}^{\top}\varPhi_{p})^{-1}\varPhi_{p}^{\top}\varPsi_{p}^{(i)}.
\end{aligned} 
\end{equation}

Similarly, to update $N^{(i+1)}$ from \eqref{algorithm2-reconstruction}, we define the following notation.
\begin{align*}
    &\mathrm{V}\times vec^{+}(N^{(i+1)})  \nonumber \\
    =~& \mathrm{I} \times vec^{+}(P^{(i+1)}) + 2\mathrm{II}_{xu}\times vec(\mathbb{B}^{(i+1)}) \nonumber \\
    &+\mathrm{III}_{\bar{u}} \times vec^{+}(\mathbb{D}^{(i+1)}) +\mathrm{VI}\times vec^{+}(R+\mathbb{D}^{(i+1)}),
\end{align*}
where
\begin{align*}
    \mathrm{V}=~&\begin{bmatrix}
        \mathbb{E}^{\mathcal{F}_t}\left[\int_t^{t+\Delta t}\bar{X}_{\mathcal{L}}(s)^{\top}ds\right] \\
        \vdots \\
        \mathbb{E}^{\mathcal{F}_t}\left[\int_{t+(l-1)\Delta t}^{t+l\Delta t}\bar{X}_{\mathcal{L}}(s)^{\top}ds\right]
    \end{bmatrix}, \\
    \mathrm{VI}=~&\begin{bmatrix}
        \mathbb{E}^{\mathcal{F}_t}\left[\int_t^{t+\Delta t}{\overline{K^{(i+1)}X_\mathcal{L}(s)}}^{\top}ds\right] \\
        \vdots \\
        \mathbb{E}^{\mathcal{F}_t}\left[\int_{t+(l-1)\Delta t}^{t+l\Delta t}{\overline{K^{(i+1)}X_\mathcal{L}(s)}}^{\top}ds\right]
    \end{bmatrix}.
\end{align*}
We denote
\[
\begin{aligned}
\varPsi_{q}^{(i+1)}\coloneqq &~\mathrm{I} \times vec^{+}(P^{(i+1)}) + 2\mathrm{II}_{xu}\times vec(\mathbb{B}^{(i+1)}) \nonumber \\
    &+\mathrm{III}_{\bar{u}} \times vec^{+}(\mathbb{D}^{(i+1)}) +\mathrm{VI}\times vec^{+}(R+\mathbb{D}^{(i+1)}).
\end{aligned}
\]
If $\mathrm{V}$ has full column rank, $N^{(i+1)}$ can be uniquely calculated as
\begin{equation} \label{eq24-a}
vec^{+}(N^{(i+1)})=({\mathrm{V}^{\top}}\mathrm{V})^{-1}{\mathrm{V}^{\top}}\varPsi_{q}^{(i+1)}.
\end{equation}

The following lemma establishes rank conditions that guarantee \eqref{eq23-a} and \eqref{eq24-a} have unique solutions.
\begin{lemma} \label{lemma3.1}
{\rm If there exists $l_0\in \mathbb{Z}^{+}$ such that, for all $l\geq l_0$,}
\begin{align}
rank(\left[\mathrm{II}_{xx},\mathrm{II}_{xu},\mathrm{III}_{\bar{u}}\right])=&~\frac{n(n+1)}{2}+mn+\frac{m(m+1)}{2}, \label{eq25a}\\
rank(\mathrm{V})=&~\frac{n(n+1)}{2}, \label{eq25b}
\end{align}
{\rm then \eqref{eq23-a} and \eqref{eq24-a} each admit a unique solution.}
\end{lemma}

\begin{proof}
First, we show that the solution to \eqref{algorithm2-correction} is unique. This amounts to proving that 
\begin{equation} \label{eqA1}
\begin{aligned}
\varPhi_{p}M &= 0 
\end{aligned}
\end{equation}
admits only the trivial solution $M=\mathbf{0}$. 

To this end, we prove by contradiction. Assume that $M=\left[{vec^+(M_1)}^{\top}, {vec(M_{2})}^{\top}, {vec^+(M_3)}^{\top}\right]^{\top}$ $\in \mathbb{R}^{\frac{n(n+1)}{2}+mn+\frac{m(m+1)}{2}}$ is a nonzero solution of \eqref{eqA1}, where ${vec^{+}(M_1)}\in \mathbb{R}^{\frac{n(n+1)}{2}}$, ${vec(M_{2})} \in \mathbb{R}^{mn}$ and ${vec^+(M_3)} \in \mathbb{R}^{\frac{m(m+1)}{2}}$. Applying $\mathrm{It\hat{o}}$'s formula to ${X_\mathcal{L}(s)}^{\top}M_{1}X_\mathcal{L}(s)$ and then taking the conditional expectation $\mathbb{E}^{\mathcal{F}_t}$ yields
\begin{align}
&\mathbb{E}^{\mathcal{F}_t}[{X_\mathcal{L}(t+\Delta t)^{\top}}M_{1}X_\mathcal{L}(t+\Delta t)-X_\mathcal{L}(t)^{\top}M_{1}X_\mathcal{L}(t)] \nonumber\\
=&~\mathbb{E}^{\mathcal{F}_t}\int_t^{t+\Delta t}{X_\mathcal{L}(s)}^{\top} \left(M_{1}(A+BK_\mathcal{T})+(A+BK_\mathcal{T})^{\top}M_{1}+(C+DK_\mathcal{T})^{\top}M_{1}(C+DK_\mathcal{T}) \right.\nonumber \\
&\left. +\int_{\mathcal{E}}(E(e)+F(e)K_\mathcal{T})^{\top}M_{1}(E(e)+F(e)K_\mathcal{T})\lambda(de)\right) X_\mathcal{L}(s)ds \nonumber\\
&+2\mathbb{E}^{\mathcal{F}_t}\int_t^{t+\Delta t}(u_\mathcal{L}(s)-K_\mathcal{T}X_\mathcal{L}(s))^{\top}\left[B^{\top}M_{1}+D^{\top}M_{1}\left(C+DK_\mathcal{T}\right) \right. \nonumber \\
&\left. +\int_{\mathcal{E}}F(e)^{\top}M_{1}\left(E(e)+F(e)K_\mathcal{T}\right)\lambda(de)\right]X_\mathcal{L}(s)ds \nonumber\\
&+\mathbb{E}^{\mathcal{F}_t}\int_t^{t+\Delta t}\left(u_\mathcal{L}(s)-K_\mathcal{T}X_\mathcal{L}(s)\right)^{\top}\left(D^{\top}M_{1}D+\int_{\mathcal{E}}F(e)^{\top}M_{1}F(e)\lambda(de)\right)\left(u_\mathcal{L}(s)-K_\mathcal{T}X_\mathcal{L}(s)\right)ds. \label{eqA3}
\end{align}

Combining equation \eqref{algorithm2-correction} and \eqref{eqA3} with the definition of $\Phi_{p}$, we obtain the following equation
\begin{equation*}
\begin{aligned}
\Phi_{p}M = \mathrm{V}\times vec^{+}(\mathcal{Y}_1) + \mathrm{II}_{xu}\times vec(\mathcal{Y}_2) + \mathrm{III}_{\bar{u}}\times vec^{+}(\mathcal{Y}_3),
\end{aligned}
\end{equation*}
where 
\begin{align}
\mathcal{Y}_1 =&-M_{1}(A+BK_\mathcal{T})-(A+BK_\mathcal{T})^{\top}M_{1}-(C+DK_\mathcal{T})^{\top}M_{1}(C+DK_\mathcal{T}) \nonumber\\
&-\int_{\mathcal{E}}(E(e)+F(e)K_\mathcal{T})^{\top}M_{1}(E(e)+F(e)K_\mathcal{T})\lambda(de) \nonumber\\
&+2{K_\mathcal{T}}^{\top}\left(B^{\top}M_{1}+D^{\top}M_{1}\left(C+DK_\mathcal{T}\right)+\int_{\mathcal{E}}F(e)^{\top}M_{1}\left(E(e)+F(e)K_\mathcal{T}\right)\lambda(de)\right) \nonumber\\
&-2K_\mathcal{T}^{\top}M_{2}-{K_\mathcal{T}}^{\top}M_{3}K_\mathcal{T}-{K_\mathcal{T}}^{\top}\left(D^{\top}M_{1}D+\int_{\mathcal{E}}F(e)^{\top}M_{1}F(e)\lambda(de)\right)K_\mathcal{T}, \label{eqA5} \\
\mathcal{Y}_2 =&~2M_{2}-2\left(B^{\top}M_{1}+D^{\top}M_{1}\left(C+DK_\mathcal{T}\right)+\int_{\mathcal{E}}F(e)^{\top}M_{1}\left(E(e)+F(e)K_\mathcal{T}\right)\lambda(de)\right) \nonumber\\
&+2\left(D^{\top}M_{1}D+\int_{\mathcal{E}}F(e)^{\top}M_{1}F(e)\lambda(de)\right)K_\mathcal{T}, \label{eqA6}\\
\mathcal{Y}_3 =&~M_{3}-D^{\top}M_{1}D-\int_{\mathcal{E}}F(e)^{\top}M_{1}F(e)\lambda(de). \label{eqA7} 
\end{align}
Then we have
\begin{equation} \label{eqA8}
\left[\mathrm{V}, \mathrm{II}_{xu}, \mathrm{III}_{\bar{u}}\right]\left(\begin{array}{cc}
vec^{+}(\mathcal{Y}_1) \\ 
vec(\mathcal{Y}_2) \\
vec^{+}(\mathcal{Y}_3)
\end{array}\right)=0.
\end{equation}

Noting that $\mathcal{Y}_1$ is symmetric, we can derive $\mathrm{V}\times vec^{+}(\mathcal{Y}_1)=\mathrm{II}_{xx}\times vec(\mathcal{Y}_1)$. Under condition \eqref{eq25a}, $\left[ \mathrm{V}, \mathrm{II}_{xu}, \mathrm{III}_{\bar{u}} \right]$ has full column rank. Hence, the unique solution to \eqref{eqA8} satisfies $vec^{+}(\mathcal{Y}_1)=\mathbf{0}$, $vec(\mathcal{Y}_2)=\mathbf{0}$ and $vec^{+}(\mathcal{Y}_3)=\mathbf{0}$. By the definition of $vec(\cdot)$ and $vec^{+}(\cdot)$, we further obtain $\mathcal{Y}_1=\mathbf{0}$, $\mathcal{Y}_2=\mathbf{0}$ and $\mathcal{Y}_3=\mathbf{0}$.
Substituting $\mathcal{Y}_2=\mathbf{0}$ and $\mathcal{Y}_3=\mathbf{0}$ into \eqref{eqA5}, the following Lyapunov equation holds.
\begin{align*}
&M_{1}(A+BK_\mathcal{T})+(A+BK_\mathcal{T})^{\top}M_{1}+(C+DK_\mathcal{T})^{\top}M_{1}(C+DK_\mathcal{T}) \nonumber \\
&+\int_{\mathcal{E}}(E(e)+F(e)K_\mathcal{T})^{\top}M_{1}(E(e)+F(e)K_\mathcal{T})\lambda(de)=0.
\end{align*}
Since $K_\mathcal{T}$ is a stabilizer, it follows from Definition \hyperref[definition2.1]{2.1} that the closed-loop system 
\begin{equation} \label{eqA10}
\left\{
\begin{aligned}
dX_\mathcal{L}(s)&=\left[A+BK_\mathcal{T}\right]X_\mathcal{L}(s)ds+\left[C+DK_\mathcal{T}\right]X_\mathcal{L}(s)dW(s)\\
&~~~+\int_{\mathcal{E}}\left[E(e)+F(e)K_\mathcal{T}\right]X_\mathcal{L}(s)\widetilde{N}(ds,de),~~s\geq t, \\
X_\mathcal{L}(t)&=x,
\end{aligned}
\right.
\end{equation}
satisfies $\lim_{s \to +\infty}\mathbb{E}[X_\mathcal{L}(s)^{\top}X_\mathcal{L}(s)]=0$. 

Applying $\mathrm{It\hat{o}}$'s formula to ${X_\mathcal{L}(s)}^{\top}M_{1}X_\mathcal{L}(s)$, for any $\tau>t$, taking the conditional expectation $\mathbb{E}^{\mathcal{F}_t}$ yields 
\begin{align}
&~\mathbb{E}^{\mathcal{F}_t}\left[{X_\mathcal{L}(\tau)^{\top}}M_{1}X_\mathcal{L}(\tau)-X_\mathcal{L}(t)^{\top}M_{1}X_\mathcal{L}(t)\right] \nonumber\\
=&~\mathbb{E}^{\mathcal{F}_t}\int_t^{\tau}{X_\mathcal{L}(s)}^{\top}\left[M_{1}(A+BK_\mathcal{T})+(A+BK_\mathcal{T})^{\top}M_{1}+(C+DK_\mathcal{T})^{\top}M_{1}(C+DK_\mathcal{T}) \right.  \nonumber \\
& \left. +\int_{\mathcal{E}}(E(e)+F(e)K_\mathcal{T})^{\top}M_{1}(E(e)+F(e)K_\mathcal{T})\lambda(de)\right]X_\mathcal{L}(s)ds, \label{eqA11} 
\end{align}
where $X_\mathcal{L}(\cdot)$ is governed by the system \eqref{eqA10}.

Letting $\tau \to \infty$, we obtain $x^{\top}M_{1}x=\lim_{\tau \to +\infty}\mathbb{E}[X(\tau)^{\top}M_{1}X(\tau)]=0$. Since $x$ can be any nonzero element in $\mathbb{R}^{n}$, it implies $M_{1}=\mathbf{0}$. Combining \eqref{eqA6} and \eqref{eqA7}, we further obtain $M_{2}=\mathbf{0}$ and $M_{3}=\mathbf{0}$. This contradicts the assumption that $M$ is nonzero. Therefore, $\Phi_{p}$ has full column rank, and \eqref{eq23-a} yields the unique solution for $P^{(i+1)}$, $\mathbb{B}^{(i+1)}$, $\mathbb{D}^{(i+1)}$. 

Second, by \cite{modares2014linear,xue2023data}, \eqref{eq25b} ensures that the coefficient matrix in \eqref{eq24-a} has full column rank. Thus $N^{(i+1)}$ is uniquely determined by \eqref{eq24-a}. This completes the proof.
\end{proof}

In summary, this section develops a model-free off-policy inverse policy iteration algorithm based on the learner's trajectories. The trajectories are generated under an initial stabilizing behavior policy, whereas the feedback policies updated during iteration are not used for data collection. This off-policy structure allows the collected trajectories to be used to construct implementable learning equations without requiring knowledge of the system parameters or jump intensity. Therefore, when the learner's data are available, this algorithm provides a complementary model-free approach for learning equivalent cost weights. The stability and convergence of the proposed algorithm are also established. 

\section{Simulation Results} \label{section5}
~~~~This section presents a numerical example to illustrate the effectiveness of data-driven Algorithms \hyperref[Algorithm 3]{1} and \hyperref[Algorithm 2]{2}. The parameter matrices of system \eqref{expert-system} are given by
\begin{equation*}
A=\begin{bmatrix}
2.0 & 2.0 \\ 
-0.9 & 0.2 
\end{bmatrix},~~
B=\begin{bmatrix}
2.1 \\ 
2.1 
\end{bmatrix}, ~~
C=\begin{bmatrix}
0.25 & 0.06 \\
0.06 & 0.25 
\end{bmatrix},~~
D=\begin{bmatrix}
0.05\\
0.04
\end{bmatrix},
\end{equation*}
\begin{equation*}
E(e)=\begin{bmatrix}
-0.35 & 0 \\
0 & -0.35 
\end{bmatrix},~~
F(e)=\begin{bmatrix}
0 \\
0
\end{bmatrix},~~\lambda = 0.9.
\end{equation*}
 
The target cost weights $N_\mathcal{T}$ and $R_\mathcal{T}$, together with the associated solution $P_\mathcal{T}$ and target control gain $K_\mathcal{T}$, are given as follows:
\begin{equation*}
N_\mathcal{T}=\left[
\begin{array}{cc}
1 & 0 \\
0 & 3
\end{array}\right],~~
R_\mathcal{T} = 5,
\end{equation*}
\begin{align*}
P_\mathcal{T}=\left[
\begin{array}{cc}
5.5969 & -0.3081 \\
-0.3081 & 2.1408
\end{array}\right],~~
K_\mathcal{T}= \left[
\begin{array}{cc}
-2.2283 & -0.7739 
\end{array}\right].
\end{align*}

The learner agent's initial cost function weight $N^{(0)}$ and prescribed $R$ are selected as
\begin{equation*}
N^{(0)}=\left[
\begin{array}{cc}
0.01 & 0 \\
0 & 0.01
\end{array}\right],~~
R=5.
\end{equation*}

A probing noise $e_\mathcal{T}$ is added to the expert input, where $e_\mathcal{T}$ is chosen as a bounded white noise signal. The limits $\mathbf{Q}^{*}$, $N^{*}$ and $K^{*}$ obtained by Algorithm \hyperref[Algorithm 3]{1} are
\begin{equation*}
\mathbf{Q}^{*}=\begin{bmatrix}
289.1077 & 100.4769 & 127.3382 \\
100.4769 & 35.9555 & 44.5142 \\
127.3382 & 44.5142 & 56.8349    
\end{bmatrix}, ~~
N^{*}= \begin{bmatrix}
5.9808 & -0.5341 \\
-0.5341 & 0.0654
\end{bmatrix}, ~~
K^{*}= \begin{bmatrix}
-2.2176 & -0.7844    
\end{bmatrix}.
\end{equation*}
We observe that although the limit $N^{*}$ is not equal to $N_\mathcal{T}$, the limit $K^{*}$ approximately equals $K_\mathcal{T}$ with an error $\|K^{*}-K_\mathcal{T}\|= 1.495\times 10^{-2}$, as further illustrated in Fig.\hyperref[fig1:sub1]{1(a)}. Moreover, Fig.\hyperref[fig1:sub2]{1(b)} and Fig.\hyperref[fig1:sub3]{1(c)} show that the learner's state and control trajectories closely follow the expert's behavior.
\begin{figure*}[t]
    \centering
    \subfloat[]{%
        \begin{minipage}[c]{0.32\textwidth}
            \centering
            \includegraphics[height=4.30cm]{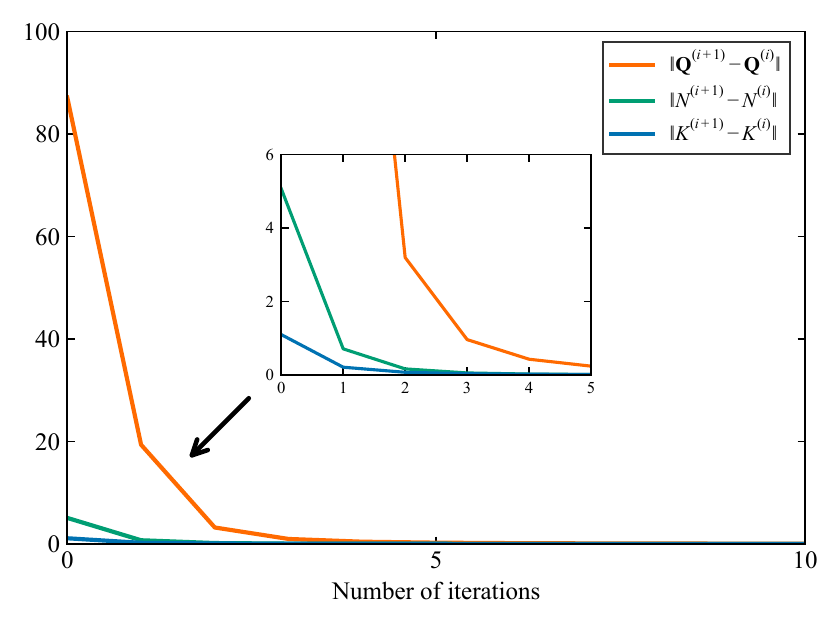}
            \label{fig1:sub1}
        \end{minipage}
    }
    \hfill
    \subfloat[]{%
        \begin{minipage}[c]{0.32\textwidth}
            \centering
            \includegraphics[height=4.4cm]{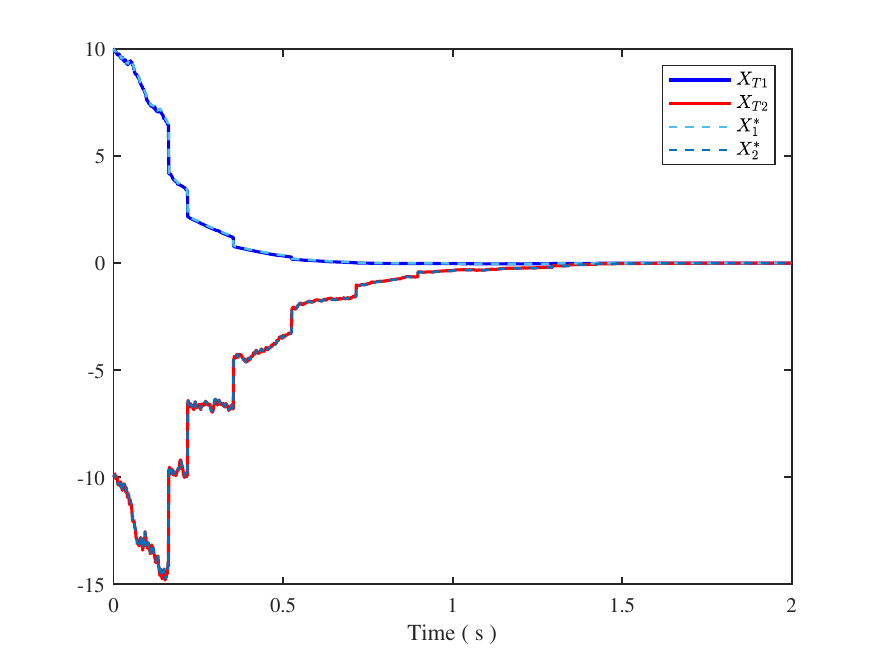}
            \label{fig1:sub2}
        \end{minipage}
    }
    \hfill
    \subfloat[]{%
        \begin{minipage}[c]{0.32\textwidth}
            \centering
            \includegraphics[height=4.4cm]{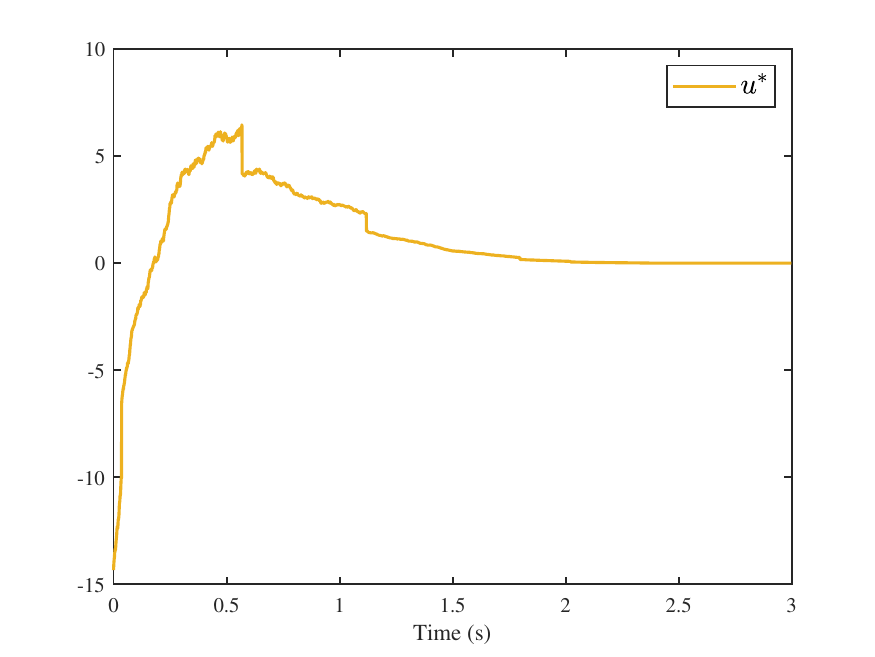}
            \label{fig1:sub3}
        \end{minipage}
    }
    \vspace{4pt}
    \caption{\protect\unboldmath All simulations using Algorithm \hyperref[Algorithm 3]{1}. (a) Convergence of $\mathbf{Q}^{(i+1)}$, $N^{(i+1)}$ and $K^{(i+1)}$. (b) State trajectories of $X_\mathcal{T}$ and $X^{*}$ using the limit $K^{*}$. (c) Control performance under the limit $K^{*}$.}
    \label{Fig 1}
\end{figure*} 

Next, we present the simulation results of Algorithm \hyperref[Algorithm 2]{2}. The initial state of system \eqref{expert-system} is $x = [10,-10]^{\top}$. To ensure that $\Phi_{p}$ and $I_{\bar{x}}$ satisfy the rank conditions in Lemma \hyperref[lemma3.1]{4.1}, the learner collects data under the initial behavior policy $u_\mathcal{L} = K_\mathcal{L}X_\mathcal{L} + e_\mathcal{L}$ on the time interval $[0,1]$, where the initial stabilizer is $K_\mathcal{L}=[-0.78 -0.86]$ and $e_\mathcal{L}$ is a probing noise. Fig.\hyperref[fig6:images]{2} illustrates that the learner's state trajectory $X_\mathcal{L} =(X_{\mathcal{L}1},X_{\mathcal{L}2})$ under the initial feedback gain $K_\mathcal{L}$ converges to zero. 
\begin{figure}[H]
\centering
\includegraphics[height=4.5cm]{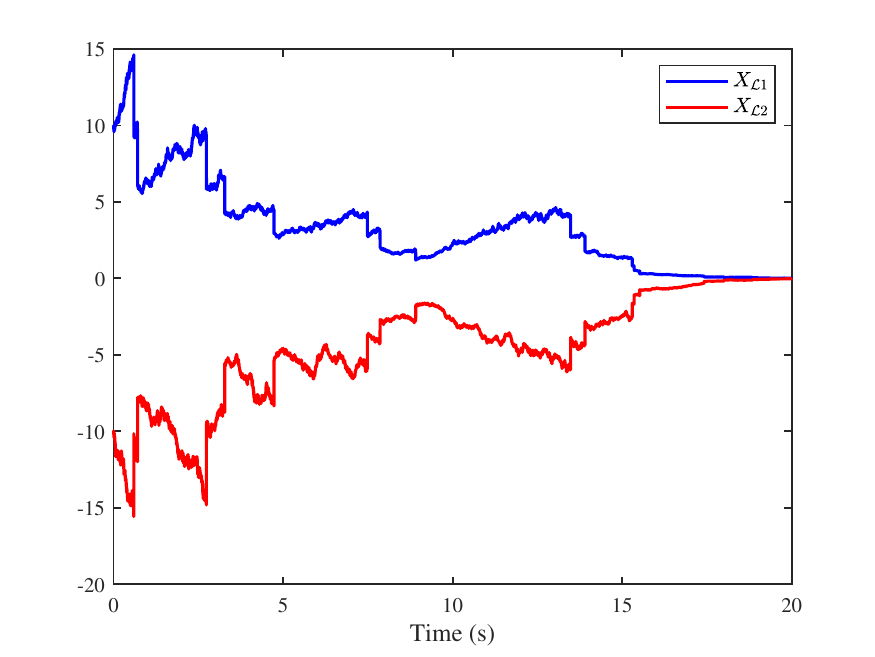}
\caption{State trajectory $X_\mathcal{L}$ running with the initial stabilizer $K_\mathcal{L}$.}
\label{fig6:images}
\end{figure}

The limits $\check{P}$, $\check{N}$, and $\check{K}$ are obtained as follows: 
\begin{equation*}
\check{P}=\begin{bmatrix}
4.7332 & 0.3964 \\
0.3964 & 1.3986
\end{bmatrix}, ~~
\check{N}= \begin{bmatrix}
4.3825 & -1.0633 \\
-1.0633 & 0.4382
\end{bmatrix}, ~~
\check{K}= \begin{bmatrix}
    -2.2279 & -0.7809
\end{bmatrix}.
\end{equation*}
The limit $\check{K}$ approximately equals $K_\mathcal{T}$ with an error $\|\check{K}-K_\mathcal{T}\|= 1.573\times 10^{-2}$, as further illustrated in Fig.\hyperref[fig3:sub1]{3(a)}. It can be observed that $\check{K}$ is very close to $K_\mathcal{T}$ while $\check{N}$ differs from $N_\mathcal{T}$. As shown in Fig.\hyperref[fig3:sub2]{3(b)} and \hyperref[fig3:sub3]{3(c)}, the learner uses Algorithm \hyperref[Algorithm 2]{2} to find an equivalent cost functional. The converged policy achieves performance comparable to that of the expert and stabilizes system \eqref{expert-system} faster than the initial behavior policy. Therefore, we conclude that both Algorithms \hyperref[Algorithm 3]{1} and \hyperref[Algorithm 2]{2} can find equivalent cost functionals whose optimal feedback gains coincide with the target gain, even though the learned cost weights may differ from the target weights.

\begin{figure*}[t]
    \vspace*{-0.50cm}
    \centering
    \subfloat[]{%
        \begin{minipage}[c]{0.32\textwidth}
            \centering
            \includegraphics[height=4.3cm]{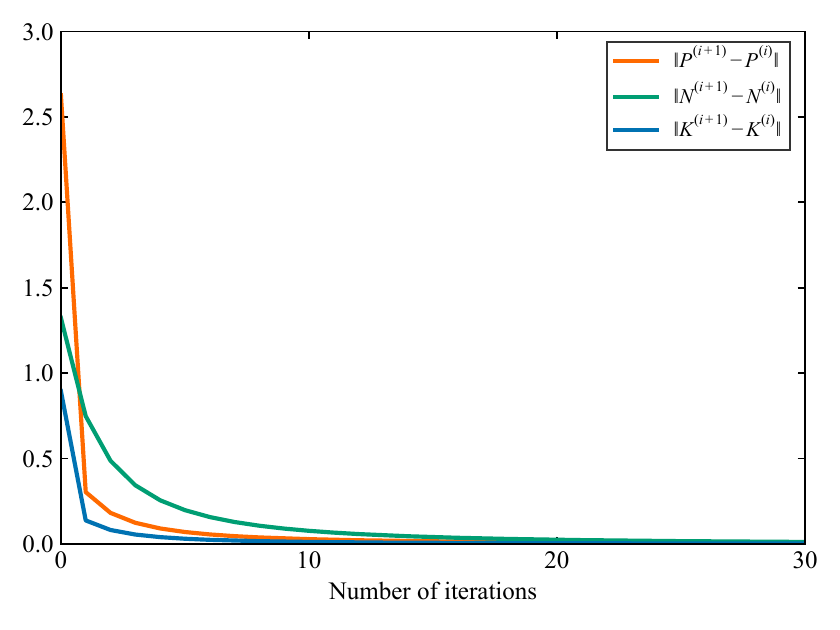}
            \label{fig3:sub1}
        \end{minipage}
    }
    \hfill
    \subfloat[]{%
        \begin{minipage}[c]{0.32\textwidth}
            \centering
            \includegraphics[height=4.4cm]{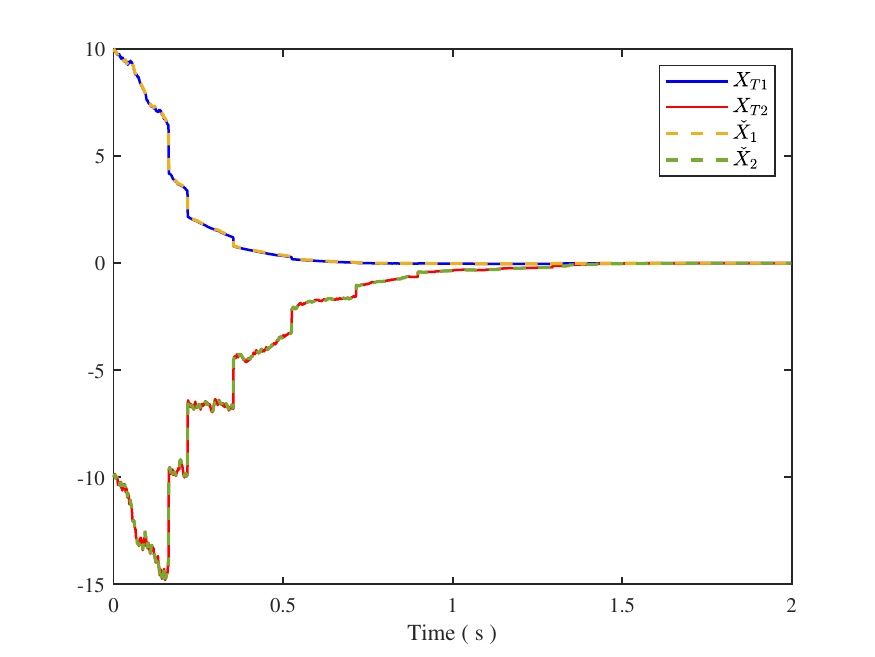}
            \label{fig3:sub2}
        \end{minipage}
    }
    \hfill
    \subfloat[]{%
        \begin{minipage}[c]{0.32\textwidth}
            \centering
            \includegraphics[height=4.4cm]{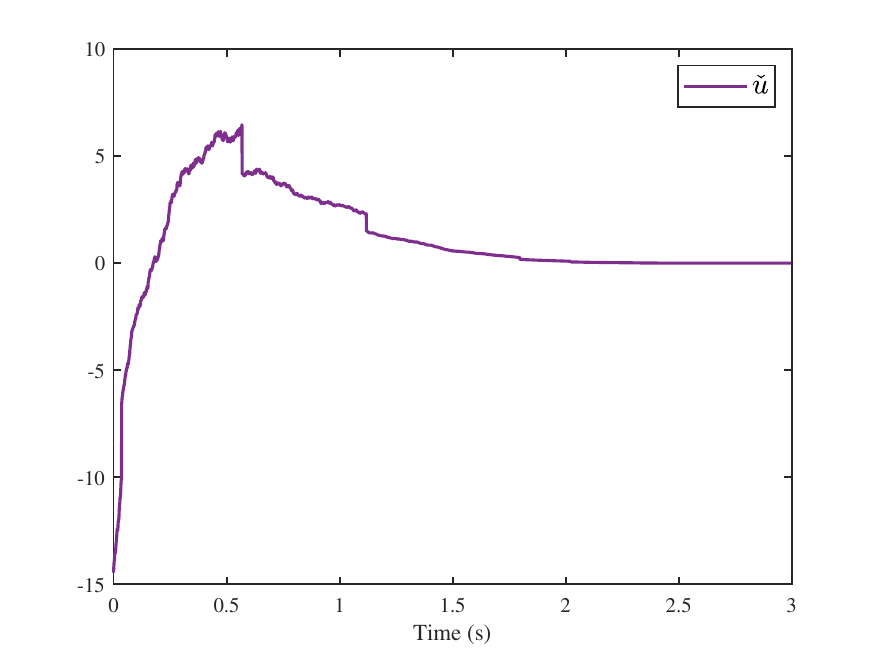}
            \label{fig3:sub3}
        \end{minipage}
    }
    \vspace{4pt}
    \caption{\protect\unboldmath All simulations using Algorithm \hyperref[Algorithm 2]{2}. (a) Convergence of $P^{(i+1)}$, $N^{(i+1)}$ and $K^{(i+1)}$. (b) State trajectories of $X_\mathcal{T}$ and $\check{X}$ using the limit $\check{K}$. (c) Control performance under the limit $\check{K}$.}
\end{figure*}

\section{Conclusion} \label{section6}
~~~~In this paper, we investigate an IOC problem for stochastic linear systems driven by both Brownian motion and Poisson jumps within an IRL framework. The objective is to find an equivalent cost functional without requiring knowledge of the system dynamics or jump intensity. Two model-free off-policy IRL algorithms are developed under different data scenarios. The inverse Q-learning algorithm uses expert demonstrations to construct data-driven Q-function equations. As a complementary framework, the model-free off-policy inverse policy iteration algorithm uses learner data collected under an initial stabilizing behavior policy to iteratively learn equivalent cost weights. In both algorithms, the behavior policy used for data collection is decoupled from the updated policies during iteration, and sufficient excitation is imposed to guarantee the required rank conditions. The stability and convergence of the proposed algorithms are rigorously established. Finally, numerical simulations shed light on the effectiveness of the proposed methods.

\bibliographystyle{elsarticle-harv}
\bibliography{IRL-arxiv}

\end{document}